\documentclass[12pt]{article}
\usepackage{graphicx}
\usepackage{amsmath,amsthm,amssymb,enumerate}
\usepackage{euscript,mathrsfs}
\usepackage{color}
\usepackage{dsfont}
\usepackage{url}
\usepackage[left=2cm,right=2cm,top=3.5cm,bottom=3.5cm]{geometry}
\usepackage{color}
\usepackage[framemethod=tikz]{mdframed}
\allowdisplaybreaks

\usepackage{soul}

\catcode`\@=11 \@addtoreset{equation}{section}

\catcode`\@=12

\newtheorem{Theorem}{Theorem}[section]
\newtheorem{Proposition}[Theorem]{Proposition}
\newtheorem{Lemma}[Theorem]{Lemma}
\newtheorem{Corollary}[Theorem]{Corollary}

\theoremstyle{definition}
\newtheorem{Definition}[Theorem]{Definition}

\newtheorem{Remark}[Theorem]{Remark}

\newcommand{\bTheorem}[1]{
	\begin{Theorem} \label{T#1} }
	\newcommand{\eT}{\end{Theorem}}

\newcommand{\bProposition}[1]{
	\begin{Proposition} \label{P#1}}
	\newcommand{\eP}{\end{Proposition}}

\newcommand{\bLemma}[1]{
	\begin{Lemma} \label{L#1} }
	\newcommand{\eL}{\end{Lemma}}

\newcommand{\bCorollary}[1]{
	\begin{Corollary} \label{C#1} }
	\newcommand{\eC}{\end{Corollary}}

\newcommand{\bRemark}[1]{
	\begin{Remark} \label{R#1} }
	\newcommand{\eR}{\end{Remark}}

\newcommand{\bDefinition}[1]{
	\begin{Definition} \label{D#1} }
	\newcommand{\eD}{\end{Definition}}

\newcommand{\pprec}{\prec \hspace{-0.2cm} \prec}

\newcommand{\tvmE}{\widetilde{\mathcal{E}}}

\newcommand{\tvm}{\widetilde{\vc{m}}}

\newcommand{\bfphi}{\boldsymbol{\varphi}}

\newcommand{\bFormula}[1]{
	\begin{equation} \label{#1}}
	\newcommand{\eF}{\end{equation}}

\newcommand{\vrn}{\vr_n}

\newcommand{\Ov}[1]{\overline{#1}}

\newcommand{\toMo}{\stackrel{\mathcal{M}}{\to}}
\newcommand{\toW}{\stackrel{\mathcal{W}}{\to}}

\newcommand{\vr}{\varrho}

\newcommand{\tvr}{\wtilde \vr}

\newcommand{\vu}{\vc{u}}
\newcommand{\vm}{\vc{m}}

\newcommand{\vc}[1]{{\bf #1}}

\newcommand{\Div}{{\rm div}_x}
\newcommand{\Grad}{\nabla_x}
\newcommand{\Dt}{\frac{\rm d}{{\rm d}t}}

\newcommand{\dx}{\,{\rm d} {x}}

\newcommand{\dt}{\,{\rm d} t }

\newcommand{\intO}[1]{\int_{\Omega} #1 \ \dx}

\newcommand{\D}{{\rm d}}

\newcommand{\ep}{\varepsilon}

\newcommand{\Td}{\mathbb{T}^d}

\newcommand{\br}{ \nonumber \\ }

\def\softd{{\leavevmode\setbox1=\hbox{d}%
		\hbox to 1.05\wd1{d\kern-0.4ex{\char039}\hss}}}
\definecolor{Cgrey}{rgb}{0.85,0.85,0.85}
\definecolor{Cblue}{rgb}{0.50,0.85,0.85}
\definecolor{Cred}{rgb}{1,0,0}
\definecolor{fancy}{rgb}{0.10,0.85,0.10}
\definecolor{amaranth}{rgb}{0.9, 0.17, 0.31}

\newcommand\Cbox[2]{%
	\newbox\contentbox%
	\newbox\bkgdbox%
	\setbox\contentbox\hbox to \hsize{%
		\vtop{
			\kern\columnsep
			\hbox to \hsize{%
				\kern\columnsep%
				\advance\hsize by -2\columnsep%
				\setlength{\textwidth}{\hsize}%
				\vbox{
					\parskip=\baselineskip
					\parindent=0bp
					#2
				}%
				\kern\columnsep%
			}%
			\kern\columnsep%
		}%
	}%
	\setbox\bkgdbox\vbox{
		\color{#1}
		\hrule width  \wd\contentbox %
		height \ht\contentbox %
		depth  \dp\contentbox
		\color{black}
	}%
	\wd\bkgdbox=0bp%
	\vbox{\hbox to \hsize{\box\bkgdbox\box\contentbox}}%
	\vskip\baselineskip%
}

\mdfdefinestyle{MyFrame}{%
	linecolor=black,
	outerlinewidth=1pt,
	roundcorner=5pt,
	innertopmargin=\baselineskip,
	innerbottommargin=\baselineskip,
	innerrightmargin=10pt,
	innerleftmargin=10pt,
	backgroundcolor=white!20!white}

\newcommand{\wtilde}{\widetilde}

\allowdisplaybreaks

\begin{document}


\title{\bf Maximal dissipation and well-posedness of the Euler system of gas dynamics}

\author{Eduard Feireisl
	\thanks{The authors gratefully acknowledge Research in Teams Fellowship
		``Entropy methods for evolutionary systems: Analysis and Numerics" of the Erwin Schr\"odinger
		International Institute for Mathematics and Physics, Vienna. \\
		The work of E.F.\ was partially supported by the
		Czech Sciences Foundation (GA\v CR), Grant Agreement
		24--11034S. The Institute of Mathematics of the Academy of Sciences of
		the Czech Republic is supported by RVO:67985840.
		E.F.\ is a member of the Ne\v cas Center for Mathematical Modelling} \,$^\clubsuit$
	\and Ansgar J\"ungel
	\thanks{{The work of A.J.\ was funded in part by the Austrian Science Fund (FWF), grant DOI 10.55776/P33010 and 10.55776/F65. This work has received funding from the European
			Research Council (ERC) under the European Union's Horizon 2020 research and innovation programme, ERC Advanced Grant NEUROMORPH, no.~101018153.}}\, $^\clubsuit$
	\and M\' aria Luk\'a\v{c}ov\'a-Medvi\softd ov\'a\thanks{The work of  M.L.-M. was supported by the Gutenberg Research College and by
		the Deutsche Forschungsgemeinschaft (DFG, German Research Foundation) -- project number 233630050 -- TRR 146 and
		project number 525800857 -- SPP 2410 ``Hyperbolic Balance Laws: Complexity, Scales and Randomness".
		She is also grateful to  the  Mainz Institute of Multiscale Modelling  for supporting her research.}\, $^\clubsuit$
}

\date{}

\maketitle

\centerline{$^*$ Institute of Mathematics of the Academy of Sciences of the Czech Republic}
\centerline{\v Zitn\' a 25, CZ-115 67 Praha 1, Czech Republic}
\centerline{feireisl@math.cas.cz}

\bigskip
\centerline{$^\dagger$Institute of Analysis and Scientific Computing, TU Wien}
\centerline{Wiedner Hauptstr. 8--10, 1040 Wien, Austria}
\centerline{juengel@tuwien.ac.at}

\bigskip
\centerline{$^\ddagger$Institute of Mathematics, Johannes Gutenberg-University Mainz}
\centerline{Staudingerweg 9, 55 128 Mainz, Germany}
\centerline{lukacova@uni-mainz.de}

\bigskip
\centerline{$^\clubsuit$Erwin Schr\"odinger  International Institute for Mathematics and Physics}
\centerline{Boltzmanngasse 9,  1090 Wien, Austria}
\newpage

\medskip

\begin{abstract}
	
We show that any dissipative (measure--valued) solution of the compressible Euler system that complies with Dafermos' criterion of
maximal dissipation is necessarily an admissible weak solution. In addition, we propose a simple, at most two step,
selection procedure to identify a unique semigroup solution in the class of dissipative solutions to the Euler system.
Finally, we introduce a refined version of Dafermos' criterion
yielding a unique solution of the problem for any finite energy initial data.
	
\end{abstract}


{\small
	
	\noindent
	{\bf 2020 Mathematics Subject Classification:}{ 35Q31, 76N10
		(primary); 35B65, 35D99, 35F50
		(secondary) }
	
	\medbreak

	
}

\section{Introduction}
\label{i}

Mathematical models of perfect fluids in continuum mechanics exhibit the
well known difficulties including formation of singularities and ill posedness
with respect to the data. An iconic example is the
\emph{barotropic Euler system} of gas dynamics describing the time evolution of the
density $\vr = \vr(t,x)$ and the velocity $\vu = \vu(t,x)$ of a
compressible perfect (inviscid) fluid:
\begin{align}
	\partial_t \vr + \Div (\vr \vu) &= 0, \label{i1}\\
\partial_t (\vr \vu) + \Div (\vr \vu\otimes \vu) + \Grad p(\vr) &=0, \label{i2}
\end{align}
where $p$ is the pressure depending solely on the fluid density, whereas the thermal effects are ignored. The fluid occupies a domain $\Omega \subset R^d$, $d=1,2,3$, and the evolution is considered in the time interval $t \in [0, \infty)$. Some relevant initial/boundary data must be prescribed to obtain, at least formally, a mathematically well posed problem. We consider the initial conditions
\begin{equation} \label{i3}
	\vr(0, \cdot) = \vr_0,\ (\vr \vu)(0, \cdot) = \vm_0,
\end{equation}
together with the impermeability boundary condition
\begin{equation} \label{i4}
\vu \cdot \vc{n}|_{\partial \Omega} = 0.
\end{equation}		

We recall the well known facts concerning solvability of the Euler system:
\begin{itemize}
\item Problem \eqref{i1}--\eqref{i4} admits a local in time classical solution provided the initial data enjoy Sobolev regularity $W^{k,2}$, $k > 1 + \frac{d}{2}$, the initial density $\vr_0$ is bounded below away from zero, and the initial velocity (momentum) $\vm_0$ satisfies the relevant compatibility
condition, see e.g. the monograph Benzoni-Gavage, Serre	\cite{BenSer}.

\item Smooth solutions develop shock type singularities in a finite time
for a vast class of initial data, see e.g. Dafermos \cite{D4a} and Smoller \cite{SMO}.

\item Singularities in the form of implosions may develop in the physically relevant 3D geometry, see Buckmaster et al.
\cite{BuCLGS}, Cao-Labora et al.\ \cite{CLGSShiSta}, Merle et al.\ \cite{MeRaRoSz},
\cite{MeRaRoSzbis}.

\end{itemize}

\subsection{Weak solutions}

To incorporate possible physically relevant singularities, the concept of \emph{weak} distributional solutions was introduced, along with an admissibility
condition
\begin{equation} \label{i5}
\partial_t E(\vr, \vu) + \Div \left[ \Big( E(\vr, \vu) + p(\vr) \Big) \vu \right] \leq 0,
	\end{equation}
where
\[
E(\vr, \vu) = \frac{1}{2} \vr |\vu|^2 + P(\vr),\ P'(\vr) \vr - P(\vr) = p(\vr)
\]
represents the energy of the system. The weak formulation of \eqref{i1}, \eqref{i2}, along with the admissibility condition (energy inequality)
\eqref{i5}, is capable to identify a unique solution in many physically relevant situations, in particular the so--called
Riemann problem in the 1-D geometry, see e.g. Bianchini, Bressan \cite{BiaBre} or the monograph by Dafermos \cite{D4a}.

The same Riemann problem considered in the multi-D setting, however, gives rise to a large variety of different but still admissible
weak solutions ``constructed'' by the method of convex integration, see e.g. Chiodaroli, De Lellis, and Kreml \cite{ChiDelKre}
among many others. Apart from the Riemann problem, the ill-posedness of the barotropic Euler system was observed in the seminal paper \cite{LS} by De Lellis and Sz\'ekehylidi.  More recently, ill posedness was shown even in the class of H\" older continuous admissible solutions, see Giri, Kwon  \cite{GirKwo}.

Despite the abundance of weak solutions for certain ``wild'' data, a simple question of \emph{existence} of an admissible weak solution for \emph{any} finite energy
initial data remains open, even in the incompressible case, cf. Wiedemann, Sz\' ekelyhidi \cite{SzeWie}, and Wiedemann \cite{Wied}. To fill this gap, in particular
when identifying the limits of numerical approximations,  a larger class of
dissipative (measure--valued) solutions is needed, see e.g. \cite{FeiLukMizSheWa}.

\subsection{Dissipative solutions}

The origin of dissipative solutions goes back to DiPerna \cite{DiP2}, see also DiPerna, Majda \cite{DiPMaj87a}, and their concept of measure--valued solution
to capture the qualitative properties of oscillatory approximations. Unfortunately, the class of Young measures and their various extensions
to include possible concentrations (see e.g. Alibert, Bouchitt\'e \cite{AliBou}) seems to be too large to give rise to a unique solution. In particular,
the unphysical ``numerical'' oscillations may create different measure--valued solutions depending on the approximation method. Instead, only certain
physically observable quantities should be invariant and recovered independently of a specific approximation process. In  \cite{BreFeiHof19},
only the expected value (with respect to the associated Young measure) of the density $\vr$, the momentum $\vm = \vr \vu$, and the scalar quantity $\mathcal{E}$ representing the
total energy were proposed as a new concept of \emph{dissipative solution} of the Euler system.

Roughly speaking, dissipative solutions solve in the distributional sense an extended Euler system:
\begin{align}
	\partial_t \vr + \Div \vm &= 0, \label{i6}\\
	\partial_t \vm + \Div \left( \frac{\vm \otimes \vm}{\vr} \right) + \Grad p(\vr) &= - \Div \mathcal{R} \label{i7}
\end{align}
with a tensor $\mathcal{R} \geq 0$ called \emph{Reynolds stress}. In addition, there is a new variable $\mathcal{E}$ called
\emph{total energy} -- a scalar quantity satisfying
\begin{equation} \label{i8}
\frac{\D }{\dt} \mathcal{E} \leq 0.
	\end{equation}
Moreover, the energy defect
\[
\mathcal{E} - \intO{ E(\vr, \vm) } \geq 0
\]
controls the amplitude of the Reynolds stress, specifically,
\begin{equation} \label{i9}
\mathcal{E} - \intO{ E(\vr, \vm) }  \geq c(p) \int_{\Ov{\Omega}} \D \ {\rm trace}[\mathcal{R}] ,
\end{equation}
where the constant $c(p) > 0$ is determined in terms of the structural properties of the equation of state $p = p(\vr)$.
The integral on the right--hand side of \eqref{i9}
is computed in terms of the measure ${\rm trace}[\mathfrak{R}]$.
We call the quantity
$\intO{ E(\vr, \vm)}$ \emph{mean energy}.
Admissible weak solutions
(with a non--increasing total energy) are therefore characterized by the equality
\[
\mathcal{E}(t) = \intO{ E(\vr, \vm)(t, \cdot) } \ \mbox{valid for a.a.}\ t \geq 0.
\]
We refer to Section \ref{sec.dissip} for a precise definition of dissipative solutions.

\begin{Remark} \label{iR1}
The terminology we use is reminiscent of the turbulence modelling.
Note, however, that the meaning of the quantities like Reynolds stress or
energy defect are rather different in the context of the (inviscid)
Euler system and the models of viscous fluids used to describe turbulence.
	\end{Remark}

\subsubsection{Comparison with other concepts of generalized solutions}

Recently, Eiter and Lasarzik \cite{EitLas1} proposed a general approach applicable to systems of conservation laws. In the particular case of the
barotropic Euler system, their concept of weak solution coincides with the dissipative solutions introduced in the previous section, see \cite[Theorem 5.8]{EitLas1}.

Inspired by the definition of variational solutions to the incompressible Euler system by Lions \cite{LI}, Lasarzik \cite{Lasar} introduced
a similar concept of solutions based on the mere satisfaction of the relative energy inequality. A more sophisticated approach of similar ``variational'' character
was proposed also by Brenier \cite{Breni}. The resulting problem is quite elegant and its solution boils down to resolving a kind of convex variational problem.
As the relative energy inequality is available also for the measure--valued/dissipative solutions of the Euler system, cf.\ e.g.\ Gwiazda, \'Swierczewska-Gwiazda, and
Wiedemann \cite{GSWW} or \cite{FeLMMiSh}, a similar approach can be developed also for the present problem. The class of possible solutions, however, is quite
large
and the desired solution of the Euler system is correctly identified only if it is smooth globally in time. As pointed out by Brenier \cite{Breni},
the resulting solutions may fail in attaining the prescribed initial data, and, in general, they may not be even (weakly) continuous with respect to time.
In particular, the solution may not coincide with the smooth one on its life--span, see the example of the Burgers equation discussed in \cite{Breni}.

\subsubsection{Basic properties of dissipative solutions}
The dissipative solutions enjoy the following properties:

\begin{itemize}
\item {\bf Global existence.} Given the initial data
\[
\vr_0, \ \vm_0, \ \mbox{and} \ \mathcal{E}_0 \geq \intO{ E(\vr_0, \vm_0) },
\]
there exists a global--in--time dissipative solution of the Euler system. The functions $\vr(t, \cdot)$, $\vm(t, \cdot)$ are weakly continuous in time and satisfy
the initial conditions in the weak sense. Moreover, the total energy $\mathcal{E}$ can be identified with a non--increasing c\` agl\` ad function of the time, with
$\mathcal{E}(0) \equiv \mathcal{E}(0-) = \mathcal{E}_0$, see \cite{BreFeiHof19}.

\item {\bf Compatibility.} If
\begin{equation} \label{E0}
\mathcal{E}_0 = \intO{ E(\vr_0, \vm_0) }
\end{equation}
and $\vr$, $\vm$ are continuously differentiable on the time interval $[0, T)$, $\vr > 0$, then $\mathcal{R} = 0$ and $\vr$, $\vm$
represent a classical solution on $[0,T)$, see e.g. \cite{FeLMMiSh}.

\item {\bf Weak--strong uniqueness.}

If \eqref{E0} holds,
and the Euler system admits a continuously differentiable (classical) solution in $[0,T)$, then $\mathcal{R} = 0$ and any dissipative solution
coincides with the classical solution in $[0,T)$.

\item {\bf Asymptotic regularity.}

We define a partial ordering of the class of dissipative solutions emanating from the same initial data,
\begin{equation} \label{i10}
(\vr^1, \vm^1, \mathcal{E}^1)  \prec 	(\vr^2, \vm^2, \mathcal{E}^2) \ \Leftrightarrow \ \mathcal{E}^1 (t) \leq \mathcal{E}^2(t) \ \mbox{for all}\ t > 0.
\end{equation}	
We say that a dissipative solution is \emph{admissible} if it is \emph{minimal} with respect to $\prec$. Admissible dissipative solutions
exist for any finite energy initial data, see \cite{BreFeiHof19}.
As shown in \cite{EF2021JEE}, the energy defect of admissible dissipative solutions vanishes in the long run, specifically,
\begin{equation} \label{i11}
\mathcal{E}(t) - \intO{ E(\vr, \vm) (t, \cdot) } \to 0 \ \Rightarrow \
\int_{\Ov{\Omega}} \D\  \left| \mathcal{R}(t, \cdot) \right|  \to 0 \ \mbox{as}\ t \to \infty.	
	\end{equation}
	\end{itemize}

Fluids arising in real--world applications are always viscous -- they dissipate mechanical energy in some sense.
The Euler system should be therefore understood as a zero viscosity limit. The same is true for
numerical approximations containing ``artificial'' numerical viscosity as a stabilizing mechanism. On the one hand, a vanishing viscosity limit is
a weak solution of the compressible Euler system only if the approximate solutions converge strongly (pointwise a.e.), see \cite{FeiHof22} or \cite[Chapter 7, Theorem 7.5]{FeLMMiSh}. In particular,
the weak convergence of approximate solutions observed for certain numerical schemes gives rise to truly dissipative solutions, meaning $\mathcal{R} \ne 0$,
see Fjordholm et al.~\cite{FKMT}, Luk\'a\v{c}ov\'a et al.~\cite{LSY}. On the other hand, as we show in the present paper, any dissipative
solution satisfying Dafermos' admissibility criterion of maximal dissipation is necessarily a weak solution. The incompatibility of weak convergence and the principle of maximal dissipation
may be seen as a paradigm shift in our understanding of what is a correct limit of numerical approximations of models involving perfect fluids, cf.\ also Chiodaroli and Kreml \cite{ChiKre}.

\subsection{Maximal dissipation principle, DiPerna's conjecture,
Dafermos' admissibility criterion}

The first group of results obtained in the present paper concerns regularity of dissipative solutions.

It may seem that the dissipative solutions may deviate largely from the
weak solutions due to a possibly large amplitude of the
Reynolds stress $\mathcal{R}$. However, as we show in Section \ref{sed}, there exist dissipative solutions with
arbitrarily small amplitude of $\mathcal{R}$. Specifically, for any
initial data and any $\delta > 0$, there exists a dissipative solution satisfying 	
\[
\int_{\Ov{\Omega}} \D \ |\mathcal{R}(t, \cdot) |  \leq \delta \ \mbox{for any}\ t \geq 0.
\]
The proof is based on suitable partial ordering of the set of dissipative solutions and an application of Zorn's lemma.

Next, we recall DiPerna's conjecture \cite{DiP2}:

\begin{quotation}

Any measure--valued solution minimal with respect to $\prec$ is in fact a weak solution of the Euler system.

\end{quotation}

\noindent The validity of the above statement for the Euler
system is still open. Motivated by the seminal work of Dafermos \cite{Dafer},
we refine the relation $\prec$ introduced in \eqref{i10} to
\begin{align}
	(\vr^1, \vm^1, \mathcal{E}^1)  &\prec_{\rm loc} 	(\vr^2, \vm^2, \mathcal{E}^2) \br &\Leftrightarrow \br (\vr^1, \vm^1, \mathcal{E}^1)(t) &= (\vr^2, \vm^2, \mathcal{E}^2)(t)
	\ \mbox{for}\ 0 \leq t \leq T, \ \mathcal{E}^1(t) < \mathcal{E}^2(t),\
	T < t < T+ \delta, \nonumber
\end{align}
for some $0 \leq T < \infty$, $\delta > 0$.  In other words, the solutions $(\vr^1, \vm^1, \mathcal{E}^1)$, $(\vr^2, \vm^2, \mathcal{E}^2)$
coincide up to the time $T$, while
$(\vr^1, \vm^1, \mathcal{E}^1)$
dissipates more energy than $(\vr^2, \vm^2, \mathcal{E}^2)$
on the interval $(T, T + \delta)$. Note carefully that
the relation $\prec_{\rm loc}$ is a strict partial order, meaning it is irreflexive, asymmetric, and transitive.

We say that a dissipative solution is \emph{maximal dissipative} if it is minimal
with respect to the ordering $\prec_{\rm loc}$. Accordingly, maximal dissipative solutions comply with Dafermos' admissibility criterion
based on maximal dissipation.
Our main result stated in Theorem \ref{dT2} below asserts that any maximal
dissipative solution is necessarily a weak solution of the Euler system with a non--increasing
total energy. This can be seen as a rigorous verification of DiPerna's conjecture {(in the framework of Dafermos' admissibility criterion)} in the specific case of the barotropic Euler system.

\subsection{Selection criteria}
\label{sc}

The second main topic discussed in the present paper is the choice of a proper selection criterion applicable to the class of dissipative solutions.
The first step in this direction was undertaken in \cite{BreFeiHof19} (cf.\ also \cite{BreFeiHof19C}), where the approach of
Cardona and Kapitanski \cite{CorKap} motivated by the original work by Krylov \cite{KrylNV} was adapted to the Euler system. The selection procedure
consists in a \emph{successive} minimization of a family of cost functionals of the type
\[
\mathcal{F}_{n,m} = \int_0^\infty \exp(-\lambda_n t) F_m(\vr, \vm, \mathcal{E}) \dt, \ \lambda_n > 0,
\]
where the family of bounded continuous functionals $F_m$ separates points in a suitable phase space. As a result, a \emph{unique} semigroup selection is identified as the asymptotic limit
of the process. The selected solutions depend in the Borel measurable way on the initial data.

The method is a bit awkward and suffers the obvious difficulties:

\begin{itemize}
\item Although the first functional may be clearly specified as
\begin{equation} \label{i12}
\mathcal{F}_{1,1} = \int_0^\infty \exp(-t) \beta (\mathcal{E}(t)) \dt,
\end{equation}
where $\beta$ is a bounded increasing function, the subsequent choice of $F_m$ as well as the exponents $\lambda_n$ is entirely arbitrary.
Obviously, different choices may give rise to different limits and there is no clear indication of possible preferences.

\item The infinite selection process is unlikely to be efficiently numerically implemented/tested.

\end{itemize}

We propose a different much simpler selection process consisting of only two steps. Similarly to \cite{BreFeiHof19}, the first step is minimizing
\[
\mathcal{F}_{1} = \int_0^\infty \exp(-t) \mathcal{E}(t) \dt.
\]	
In comparison with \eqref{i12}, the function $\beta$ is missing, in particular, the functional is not bounded. In order to select a measurable minimizer,
a slight change of the solution space with respect to \cite{BreFeiHof19} is needed. Note that $\mathcal{F}_1$ can be obviously written
in the form
\[
\mathcal{F}_{1} =  \int_0^\infty \exp(-t) \Big[ \mathcal{E}(t) - \intO{E(\vr, \vm)(t, \cdot)} \Big] \dt +  \int_0^\infty \exp(-t) E(\vr, \vm) (t, \cdot) \dt
\]
- the sum of the (weighted) time integral of the energy defect and the mean energy.
There is a number of arguments discussed already in \cite[Section 5.1]{BreFeiHof19} why $\mathcal{F}_1$ should be the first to minimize. In particular, the selected solutions are necessarily admissible (minimal) with respect to the relation $\prec$ introduced in \eqref{i10}.

We conclude the selection process by the second step, namely minimizing
\[
\mathcal{F}_2 = \int_0^\infty \exp(-t) \intO{ F(\vr, \vm, \mathcal{E})(t) } \dt,
\]
where
\[
F: R \times R^d \times R \to [0, \infty)
\]
is a suitable strictly convex function. We may call $F$ ``entropy'' associated to the problem. In contrast to
\cite{BreFeiHof19}, the functional
\[
(\vr, \vm, \mathcal{E}) \mapsto \intO{ F(\vr, \vm, \mathcal{E})(t) }
\]
is neither bounded nor continuous with respect to the weak topologies pertinent to the solution space. Accordingly, the main new issue to be properly addressed is
the Borel measurability of the selected minimizer with respect to the data. While measurability of the solution sets in \cite{BreFeiHof19}
was reduced to measurability of the solution mapping in the Hausdorff topology on compact subsets of the phase space, we have to use
the Wijsman topology on convex closed sets of suitable $L^p$-spaces. Fortunately, under certain conditions, convergence in the latter is equivalent to
the so--called Mosco convergence of closed convex sets suitable to establish continuity of the minimizers.

Once measurability is established, step 2 yields a minimizer of $\mathcal{F}_2$. As the cost functional is strictly convex, this concludes the selection process. Here, convexity of the set of all dissipative solution
emanating from the same initial data plays a crucial role. In sharp
contrast with \cite{BreFeiHof19}, ambiguity of the result
is only due to the choice of the entropy $F$ -- a topic subject to future discussion.

As shown in Theorem \ref{sT1},
the above selection process gives rise to a unique solution in the class of admissible dissipative solutions that enjoys the semigroup property. We call this solution a {\em semigroup solution}.\label{semigroup}

\subsection{Maximal dissipation principle revisited}

Motivated by the selection procedure we introduce a concept of \emph{absolute energy minimizer}
\[
(\underline{\vr}, \underline{\vm}, \underline{\mathcal{E}})
\]
emanating from the initial data $(\vr_0, \vm_0, \mathcal{E}_0)$, and satisfying
\[
\int_0^\infty \exp(-\lambda t) \underline{\mathcal{E}}(t) \dt \leq  \int_0^\infty \exp(-\lambda t) \mathcal{E}(t) \dt \
\mbox{for all}\ \lambda \geq \underline{\lambda}
\]
for any other solution $(\vr, \vm, \mathcal{E})$ starting from the same initial data, where
$\underline{\lambda} = \underline{\lambda} (\vr, \vm, \mathcal{E})$ depends on the solution $(\vr, \vm, \mathcal{E})$.

We show that for each initial data there exists at most one absolute energy minimizer. Moreover, any absolute energy minimizer is minimal
with respect to the order $\prec_{\rm loc}$; whence a weak solution of the Euler system with a non--increasing energy profile. Finally, we observe that
any \emph{global} minimizer with respect to the order $\prec_{\rm loc}$ must be an absolute energy minimizer. This simple but quite important fact provides
a local admissibility criterion applicable on any compact time interval
$[0,T]$ \emph{independent} of the behaviour of solutions in the long run.

\vglue 1 cm

The paper is organized as follows. In Section \ref{w}, we introduce the class of dissipative solutions to the barotropic Euler system and review their basic
properties. In Section \ref{d}, we state and prove our main results concerning the existence of dissipative solutions with
arbitrarily small Reynolds stress (Theorem \ref{dT1} and Corollary \ref{dC1}). Moreover, we show that any maximal dissipative solution is
necessarily a weak solution (Theorem \ref{dT2}). The selection process to identify a unique semigroup solution is discussed in Section \ref{ss}.
In Section \ref{A}, we introduce the class of absolute energy minimizers and discuss their basic properties.

\section{Dissipative solutions to the barotropic Euler system}
\label{w}

In this section, we introduce the class of dissipative solutions to the Euler system.
For definiteness, we restrict ourselves to the isentropic state equation
\begin{equation} \label{w1}
	p(\vr) = a \vr^\gamma,\ a > 0,\ \gamma > 1.
\end{equation}
Accordingly, the corresponding pressure potential is given by
\[
	P(\vr) = \frac{a}{\gamma - 1}\vr^\gamma,\ P'(\vr) \vr - P(\vr) = p(\vr).		
\]
The energy function is defined as
\begin{equation} \label{w2}
	E(\vr, \vm) = \left\{ \begin{array}{l} \frac{1}{2} \frac{|\vm|^2}{\vr} + P(\vr) \ \mbox{if}\ \vr > 0,\\ 0 \ \mbox{if}\ \vr = 0,\ \vm = 0, \\
	\infty \ \mbox{otherwise} \end{array} \right. ,
\end{equation}		
Note that $E: R^2 \to [0, \infty]$ is a convex l.s.c. function, strictly convex on its domain.
Finally, we suppose that $\Omega \subset R^d$ is a bounded Lipschitz domain.

\subsection{Definition of dissipative solutions}
\label{sec.dissip}

The initial data $(\vr_0, \vm_0, \mathcal{E}_0)$ belong to the class
\begin{equation} \label{w7}
	\mathcal{D} = \left\{ (\vr_0, \vm_0) \ \mbox{measurable in}\ \Omega,  \
	\mathcal{E}_0 \in [0, \infty) \Big|\
	\intO{ E(\vr_0, \vm_0)  } \leq \mathcal{E}_0 \right\}.
\end{equation}
As explained in \cite{BreFeiHof19}, the total energy $\mathcal{E}$,
with its initial value $\mathcal{E}_0$,
is formally added as a ``new'' state variable in the definition of dissipative solutions.

\begin{mdframed}[style=MyFrame]

\begin{Definition}[{\bf Dissipative solution}] \label{wD1}

Let $\Omega \subset R^d$ be a bounded Lipschitz domain.
A trio $(\vr, \vm, \mathcal{E})$ is called \emph{dissipative solution} of the Euler system \eqref{i1}--\eqref{i4} in $[0, \infty) \times \Omega$ if the following holds:
\begin{enumerate}
\item{\bf Regularity.}	
\begin{align}
	\vr &\in C_{\rm loc}([0, \infty); W^{-\ell,2}(\Omega)),\ \vr(0, \cdot) = \vr_0 , \br	
	\vm &\in C_{\rm loc}([0, \infty); W^{-\ell,2}(\Omega; R^d)),\ \vm(0, \cdot) = \vm_0,\ \ell > d, \br
	\mathcal{E}: [0, \infty) &\to [0, \infty)
	- \mbox{a c\` agl\` ad function},\ \mathcal{E}(0) = \mathcal{E}_0.
	\label{w3}
\end{align}	
\item{\bf Field equations.}
The equation of continuity \eqref{i1} is satisfied in the weak sense,
\begin{equation} \label{w10}
	\int_{\tau_1}^{\tau_2} \intO{ \Big[ \vr \partial_t \varphi + \vm \cdot \Grad \varphi \Big] } \dt = \left[ \intO{ \vr \varphi } \right]_{t = \tau_1}^{t = \tau_2},\ \vr(0, \cdot) = \vr_0,
\end{equation}
for any $\varphi \in C^1_c([0, \infty) \times \Ov{\Omega})$, $0 \leq \tau_1 \leq \tau_2$.

The momentum equation \eqref{i2} is augmented by the action of the \emph{Reynolds stress}:
\begin{align}
	\int_{\tau_1}^{\tau_2} &\intO{ \Big[ \vm \cdot \partial_t \bfphi + \mathds{1}_{\vr > 0}
		\frac{\vm \otimes \vm}{\vr} : \Grad \bfphi + p(\vr) \Div \bfphi \Big] } \br
	&= - \int_{\tau_1}^{\tau_2} \int_{\Ov{\Omega}} \Grad \bfphi : \D {\mathcal{R}} \dt +  \left[ \intO{ \vm \cdot \bfphi } \right]_{t = \tau_1}^{t = \tau_2},\
	\vm(0, \cdot) = \vm_0
	\label{w11}
\end{align}
for any $\bfphi \in C^1_c([0, \infty) \times \Ov{\Omega}; R^d)$, $\bfphi \cdot \vc{n}|_{\partial \Omega} = 0$, $0 \leq \tau_1 \leq \tau_2$.	

The Reynolds stress tensor $\mathcal{R}$ belongs to the class
\begin{equation} \label{w8}
	\mathcal{R} \in L^\infty_{\rm weak-(*)} \Big(0, \infty; \mathcal{M}^+(\Ov{\Omega}; R^{d \times d}_{\rm sym}) \Big),
\end{equation}
where the symbol $\mathcal{M}^+(\Ov{\Omega}; R^{d \times d}_{\rm sym})$ denotes the set of all positively semi--definite
matrix valued measures on $\Ov{\Omega}$.

\item{\bf Total energy.} The total energy $\mathcal{E}$ is non--increasing
in the time variable, specifically, 
\begin{equation} \label{w12}
	\int_0^\infty \mathcal{E} \partial_t \psi \dt \geq
	- \psi(0) \mathcal{E}_0
\end{equation}
for any $\psi \in C^1_c[0, \infty)$, $\psi \geq 0$.

\item {\bf Compatibility of energy defect and Reynolds stress.}
The \emph{energy defect}
\[
D_{\mathcal{E}}(\tau) = \left( \mathcal{E}(\tau+) - \intO{ E(\vr, \vm)(\tau, \cdot) } \right),\ \tau \geq 0
\]
controls the amplitude of the Reynolds stress,
\begin{equation} \label{w9}
	D_{\mathcal{E}}(\tau) \equiv \left( \mathcal{E}(\tau+) - \intO{ E(\vr, \vm)(\tau, \cdot) } \right) \geq	
	r(d,\gamma) \int_{\Ov{\Omega}} 1 \	\D \ {\rm trace}[\mathcal{R} (\tau, \cdot)]
\end{equation}
for a.a. $\tau > 0$, where
\[
r(d,\gamma) = \min \left\{ \frac{1}{2}, \frac{d \gamma}{\gamma - 1} \right\}.
\]
\end{enumerate}	
	
\end{Definition}	

\end{mdframed}

\begin{Remark} \label{wR1}
	The specific form of the constant $r$ in \eqref{w9} is pertinent to the isentropic equation of state. In the general case, we may consider
	\begin{equation} \label{S2}
		\begin{split}
			&p \in C[0, \infty) \cap C^2(0, \infty),\
			p(0) = 0, \ p'(\vr) > 0 \ \mbox{for}\ \vr > 0; \\
			&\mbox{the pressure potential}\ P
			\ \mbox{determined by}\ P'(\vr) \vr - P(\vr) = p(\vr)\ \mbox{satisfies}\
			P(0) = 0,\\ &\mbox{and}\ P - \underline{a} p,\ \Ov{a} p - P\ \mbox{are convex functions
				for certain constants}\ \underline{a} > 0, \ \Ov{a} > 0,
		\end{split}
	\end{equation}
	see \cite{AbbFeiNov}. The constant $r$ in \eqref{w9} is then determined solely
	in terms of $\underline{a}$, $\Ov{a}$, and the dimension $d$.

\end{Remark}

As already pointed out the density - momentum component $(\vr, \vm)$ of a dissipative solution is nothing other than the expected value of a Young measure for ``conventional'' measure--valued solutions. The total energy
$\mathcal{E}$ is the standard energy $\intO{ E(\vr, \vm)}$ augmented by the energy defect, where the latter corresponds to the sum of oscillation and concentration defects in the measure--valued formulation.

Note that \eqref{w12}, \eqref{w9} imply that
\begin{align}
\vr(t, \cdot) &\geq 0 \ \mbox{a.a. in}\ \Omega,\ \vm(t, \cdot) = 0 \ \mbox{a.a. on the vacuum set}\
\{ \vr(t, \cdot) = 0 \} \ \mbox{for any}\ t \geq 0,\label{w5} \\
\vr &\in C_{\rm weak}([0, \infty); L^\gamma(\Omega)),\
\vm \in C_{\rm weak}([0, \infty); L^{\frac{2 \gamma}{\gamma + 1}}(\Omega; R^d)). \label{w6}	
	\end{align}

%
%
%
%
%
%
%
%
%
	
\subsubsection{Rate of energy dissipation, maximal dissipative solutions}
\label{nte}

Let the data $(\vr_0, \vm_0, \mathcal{E}_0)$ be given. We denote
\[
\mathcal{U}[\vr_0, \vm_0, \mathcal{E}_0] =
\left\{ (\vr, \vm, \mathcal{E})\ \Big|\
(\vr, \vm, \mathcal{E}) \ \mbox{a dissipative solution},\
\vr(0, \cdot) = \vr_0,\ \vm(0, \cdot) = \vm_0,\ \mathcal{E}(0) = \mathcal{E}_0 \right\}
\]
the set of all dissipative solutions on $[0, \infty)$ emanating from the data $(\vr_0, \vm_0, \mathcal{E}_0)$.

To compare the rate of energy dissipation, we introduce two order relations
already discussed in Section \ref{i}. The first one introduced by DiPerna
\cite{DiP2} reads:
\begin{equation} \label{DiP}
(\vr^1, \vm^1, \mathcal{E}^1) \prec
(\vr^2, \vm^2, \mathcal{E}^2) \ \Leftrightarrow \
\mathcal{E}^1 \leq \mathcal{E}^2.
\end{equation}	

\begin{mdframed}[style=MyFrame]

\begin{Definition} [{\bf Admissible dissipative solutions}] \label{wD3}
	
We say that a dissipative solution
\[
(\vr, \vm, \mathcal{E}) \in \mathcal{U}[\vr_0, \vm_0, \mathcal{E}_0]
\]	
is \emph{admissible} if it is minimal with respect to $\prec$ in
$\mathcal{U}[\vr_0, \vm_0, \mathcal{E}_0]$.
	
\end{Definition}	

\end{mdframed}
\noindent

For admissible dissipative solutions, the total energy $\mathcal{E}$ is uniquely determined by $(\vr, \vm)$. Specifically, the following holds.

\begin{Proposition}\label{mT10}

Let $(\vr^i, \vm^i, \mathcal{E}^i)$ be two admissible solutions
emanating from the data $(\vr^i_0, \vm^i_0, \mathcal{E}^i_0)$, $i=1,2$.
Suppose, \[
(\vr^1, \vm^1)(t, \cdot) = (\vr^2, \vm^2)(t, \cdot)
\ \mbox{for any}\ t \in [T, \infty), \quad T\geq 0.
\]
Then
\[
\mathcal{E}^1(t) = \mathcal{E}^2(t) \ \mbox{for any}\ t \in [T, \infty).
\]
	
\end{Proposition}

\begin{proof}
	
Denote
\[
\underline{\mathcal{E}} (t) =
\min \left\{ \mathcal{E}^1(t), \mathcal{E}^2(t) \right\} ,\ t \in [T, \infty).
\]
It is easy to check that the solutions
\[
(\vr^i, \vm^i, \underline{\mathcal{E}}^i),\ i = 1,2
\]
where
\[
\underline{\mathcal{E}}^i(t) = \left\{
\begin{array}{l}
\mathcal{E}^i (t) \ t \in [0,T), \\
\underline{\mathcal{E}}(t),\ t \in [T, \infty)
\end{array}
\right. , \ i=1,2,
\]
are again dissipative solutions with the modified Reynolds stress
\[
\mathcal{R}^i(t, \cdot) = \mathds{1}_{ \mathcal{E}^1(t)
> \mathcal{E}^2(t) } \mathcal{R}^2 (t) + \mathds{1}_{ \mathcal{E}^2(t)
\geq \mathcal{E}^2(t) } \mathcal{R}^1(t),\ t \in [T, \infty),\
\mathcal{R}^i(t) \ \mbox{ for } t \in [0,T), \ i = 1,2.
\]
However, as both solutions are minimal with respect to $\prec$, we conclude
\[
\mathcal{E}^i(t) = \underline{\mathcal{E}}(t),\ t \in [T, \infty),\ i =1,2.
\]
	
\end{proof}

DiPerna's conjecture stated in \cite[Section 6, Part (b)]{DiP2} asserts that
``\emph{any admissible dissipative solution is a weak solution}''. To the best of our knowledge this conjecture is still not proven. Moreover, its formulation is not ``deterministic'' as it depends on the behaviour of solutions on the whole interval $(0, \infty).$

Motivated by Dafermos \cite{Dafer}, we introduce a local version of
\eqref{DiP},
\begin{align} \label{Dafer}
(\vr^1, \vm^1, \mathcal{E}^1) &\prec_{\rm loc}
(\vr^2, \vm^2, \mathcal{E}^2)\br &\Leftrightarrow \br
\mbox{there exists}\ T &\geq 0 \ \mbox{such that}\
(\vr^1, \vm^1, \mathcal{E}^1)(\tau, \cdot) =
(\vr^2, \vm^2, \mathcal{E}^2)(\tau, \cdot) \ \mbox{for any}\ \tau \in [0,T],\br
\mathcal{E}^1(\tau) &< \mathcal{E}^2(\tau) \ \mbox{for}\
\tau \in (T, T+ \delta) \ \mbox{for some}\ \delta > 0.
\end{align}	

\begin{mdframed}[style=MyFrame]
	
\begin{Definition}[{\bf Maximal dissipative solutions}] \label{wD4}
We say that a dissipative solution \\ $(\vr, \vm, \mathcal{E}) \in \mathcal{U}[\vr_0, \vm_0, \mathcal{E}_0]$ is \emph{maximal dissipative}	
if it is minimal with respect to the order $\prec_{\rm loc}$ in
$\mathcal{U}[\vr_0, \vm_0, \mathcal{E}_0]$.
\end{Definition}		
	
\end{mdframed}	

One of the main results of this paper -- Theorem \ref{dT2} --
is a rigorous proof of DiPerna's conjecture in Dafermos' framework. Specifically, we show that any maximal dissipative solution is necessarily a weak solution of the Euler system.

\subsection{Properties of dissipative solutions}
\label{s}

We start by introducing the necessary function spaces framework. In accordance with \eqref{w9}, the dissipative solutions belong to the convex set
\begin{equation} \label{s1}
\mathcal{D}= \left\{ (\vr, \vm, \mathcal{E}) \ \Big| \ \vr, \vm \ \mbox{measurable in}\ \Omega,\ \mathcal{E} \geq 0
\ \intO{ E(\vr, \vm)(\tau) } \leq \mathcal{E}(\tau+) \ \mbox{for all}\ \tau \geq 0 \right\}.
\end{equation}

Next, we introduce the trajectory space $\mathcal{T}$, in which solutions live. We consider two different topologies on the trajectory space:
\begin{equation} \label{s2}
	\mathcal{T}_{\rm weak} = C_{\rm loc}([0, \infty); W^{-\ell,2}(\Omega)) \times C_{\rm loc}([0, \infty); W^{-\ell,2}(\Omega;R^d)) \times L^q_{\omega}[0, \infty),
\end{equation}	
and
\begin{equation} \label{s3}
	\mathcal{T}_{\rm strong} = L^q_{\omega}([0, \infty); L^q(\Omega)) \times L^{q}_{\omega}([0, \infty); L^{q}(\Omega;R^d)) \times L^q_{\omega }[0, \infty),
\end{equation}	
$1 < q \leq \frac{2 \gamma}{\gamma + 1} < \gamma$,
where $\omega(t) = \exp(-t)$ is an exponential weight function. Note that the weighted Lebesgue measure $\omega \Dt \otimes \dx$ considered on $[0,\infty) \times \Omega$ is of finite measure. The space $\mathcal{T}_{\rm weak}$ is a separable, complete metric space, while $\mathcal{T}_{\rm strong}$ is a uniformly convex, separable, reflexive Banach space. In particular, both topologies are Polish.

\subsubsection{Global existence}

Given the data $(\vr_0, \vm_0, \mathcal{E}_0) \in \mathcal{D}$ we introduce the solution set
\begin{align}
\mathcal{U}&[\vr_0, \vm_0, \mathcal{E}_0]\br &=
\Big\{
(\vr, \vm, \mathcal{E}) \ \mbox{a dissipative solution with the data}
\ (\vr_0, \vm_0, \mathcal{E}_0) \ \mbox{defined for all}\ t \geq 0 \Big\} \br &\subset \mathcal{T}_{\rm weak} \cap \mathcal{T}_{\rm strong}.
\label{s4}
\end{align}
As shown in \cite[Section 3.1]{BreFeiHof19}, the set
$\mathcal{U}[\vr_0, \vm_0, \mathcal{E}_0]$ is non--empty, meaning the dissipative solution exists for any data $(\vr_0, \vm_0, \mathcal{E}_0) \in \mathcal{D}$.

\begin{Remark} \label{Rs1}
	As a matter of fact, the existence in \cite{BreFeiHof19} was established for periodic boundary conditions. The extension to the impermeability conditions
	is, however, straightforward (see also \cite{BreFeiHof19C}).
\end{Remark}

\subsubsection{Convexity of the solution set}

We start with a crucial observation.
	
\begin{Lemma}[{\bf Convexity}] \label{Ls1}
For any initial data $(\vr_0, \vm_0, \mathcal{E}_0) \in \mathcal{D}$, the solution set
$\mathcal{U}[\vr_0, \vm_0, \mathcal{E}_0]$ is convex.
\end{Lemma}	

\begin{proof}
	
Let
\[
(\vr, \vm, \mathcal{E}) = \lambda (\vr^1, \vm^1, \mathcal{E}^1) + (1 - \lambda) (\vr^2, \vm^2, \mathcal{E}^2),\ \lambda \in [0,1],
\]
where
\[
(\vr^i, \vm^i, \mathcal{E}^i) \in \mathcal{U}[\vr_0, \vm_0, \mathcal{E}_0],\
i =1,2.
\]	

Obviously, the trio $(\vr, \vm, \mathcal{E})$ belongs to the regularity class \eqref{w3}, and the equation of continuity \eqref{w10}
as well as the energy inequality \eqref{w12} are satisfied.

As for the momentum balance \eqref{w11}, we easily deduce
\[
	\int_0^\tau \intO{ \Big[ \vm \cdot \partial \bfphi + \mathds{1}_{\vr > 0}
		\frac{\vm \otimes \vm}{\vr} : \Grad \bfphi + p(\vr) \Div \bfphi \Big] }
	= - \int_0^\tau \int_{\Ov{\Omega}} \Grad \bfphi : \D {\mathcal{R}} \dt +  \left[ \intO{ \vm \cdot \bfphi } \right]_{t = 0}^{t = \tau}
\]
for any $\bfphi \in C^1_c([0, \infty) \times \Ov{\Omega}; R^d)$, $\bfphi \cdot \vc{n}|_{\partial \Omega} = 0$, with a new Reynolds stress
\begin{align}
\mathcal{R} &= \lambda \mathcal{R}^1  + (1 - \lambda) \mathcal{R}^2 +
\lambda \mathds{1}_{\vr^1 > 0}
\frac{\vm^1 \otimes \vm^1}{\vr^1} + (1 - \lambda) \mathds{1}_{\vr^2 > 0}
\frac{\vm^2 \otimes \vm^2}{\vr^2} \br
&- \mathds{1}_{[\lambda \vr_1 + (1 - \lambda) \vr_2]>0} \frac{(\lambda \vm^1 +
	(1 - \lambda) \vm^2 ) \otimes (\lambda \vm^1 +
	(1 - \lambda) \vm^2 ) }{\lambda \vr^1 + (1 - \lambda) \vr^2}
 \br &+ \left[ \Big( \lambda p(\vr^1) + (1 - \lambda) p(\vr^2) \Big) - p \left(\lambda \vr^1 + (1 - \lambda) \vr^2 \right) \right] \mathbb{I}.
\label{s5}
\end{align}
It is easy to check
\begin{align}
&\left[
	\lambda \mathds{1}_{\vr^1 > 0}
	\frac{\vm^1 \otimes \vm^1}{\vr^1} + (1 - \lambda) \mathds{1}_{\vr^2 > 0}
	\frac{\vm^2 \otimes \vm^2}{\vr^2} \right. \br
&\quad \left. - \mathds{1}_{\lambda \vr_1 + (1 - \lambda) \vr_2 > 0} \frac{(\lambda \vm^1 +
	(1 - \lambda) \vm^2 ) \otimes (\lambda \vm^1 +
	(1 - \lambda) \vm^2 ) }{[\lambda \vr^1 + (1 - \lambda) \vr^2]> 0}  \right]: (\xi \otimes \xi) \br &\quad =
	\lambda \mathds{1}_{\vr^1 > 0} \frac{|\vc{m}^1 \cdot \xi|^2}{\vr^1} +
	(1 - \lambda) \mathds{1}_{\vr^2 > 0}	\frac{|\vc{m}^2 \cdot \xi|^2}{\vr^2} - \mathds{1}_{[\lambda \vr^1 + (1 - \lambda) \vr^2]> 0} \frac{|(\lambda \vm^1 +
		(1 - \lambda) \vm^2 )) \cdot \xi|^2 }{\lambda \vr^1 + (1 - \lambda) \vr^2}
	\nonumber
\end{align}	
Using the fact that $\vm_i = 0$ whenever $\vr_i = 0$, convexity of the function
\[
(\vr, \vm) \mapsto \frac{|\vm \cdot \xi|^2}{\vr},
\]
and convexity of $p$,
we conclude that the new Reynolds stress is positively semi--definite, meaning
\eqref{w8} holds.

Finally, by the same token,
\begin{align}
\mathcal{E}& - \intO{E(\vr, \vm)} = \lambda \mathcal{E}^1 -
\lambda \intO{ E(\vr^1, \vm^1) } + (1 - \lambda) \mathcal{E}^2 -
(1 - \lambda) \intO{ E(\vr^2, \vm^2)}\br &+
\lambda \intO{ E(\vr^1, \vm^1) } + (1 - \lambda) \intO{ E(\vr^2, \vm^2)}
- \intO{ E\Big( \lambda \vr^1 + (1 - \lambda) \vr^2), \lambda \vm^1 + (1 - \lambda )\vm^2     \Big) }
\br
&\geq \lambda r(d,\gamma) \int_{\Ov{\Omega}} \D \ {\rm trace}\ [\mathcal{R}^1] \dx
+ (1 - \lambda) r(d,\gamma) \int_{\Ov{\Omega}} \D \ {\rm trace} \ [\mathcal{R}^2] \dx
\br
&+
\lambda \intO{ E(\vr^1, \vm^1) } + (1 - \lambda) \intO{ E(\vr^2, \vm^2)}
- \intO{ E\Big( \lambda \vr^1 + (1 - \lambda) \vr^2), \lambda \vm^1 + (1 - \lambda )\vm^2     \Big) }
\br
&\geq r(d, \gamma) \int_{\Ov{\Omega}} \ \D \ {\rm trace}[\mathcal{R}],
\nonumber
\end{align}
which yields the compatibility condition \eqref{w9}.

	\end{proof}

Summarizing, we may infer that
the solution set $\mathcal{U}[\vr_0, \vm_0, \mathcal{E}_0]$ is:
\begin{itemize}
	\item non--empty, convex, compact in $\mathcal{T}_{\rm weak}$;
	\item non--empty, convex, closed and bounded in $\mathcal{T}_{\rm strong}$.
\end{itemize}
	
\subsubsection{Topological properties of the solution set}

It is convenient to consider the data space $\mathcal{D}$ endowed with the (strong) topology of the Hilbert space
\begin{equation} \label{s6}
\mathcal{D} \subset
W^{-\ell,2}(\Omega) \times W^{-\ell,2}(\Omega; R^d) \times R,\ \ell > d.
\end{equation}
More specifically, the data space $\mathcal{D}$ is a convex, locally compact
subset of the aforementioned Hilbert space.

We recall the following result proved in \cite[Section 3.1]{BreFeiHof19}.

\begin{Proposition}[{\bf Sequential stability}] \label{sP1}
Consider a sequence of data
\[
(\vr_{0,n}, \vm_{0,n}, \mathcal{E}_{0,n}) \to
(\vr_0, \vm_0, \mathcal{E}_0) \ \mbox{in}\ \mathcal{D}\ \mbox{as}\ n \to \infty.
\]	
Let
\[
(\vr_n, \vm_n, \mathcal{E}_n) \in \mathcal{U} [\vr_{0,n}, \vm_{0,n}, \mathcal{E}_{0,n}] \ \mbox{for}\ n =1,2,\dots.
\]	

Then there is a subsequence such that
\[
(\vr_{n_k}, \vm_{n_k}, \mathcal{E}_{n_k}) \to (\vr, \vm, \mathcal{E})
\ \mbox{in}\ \mathcal{T}_{\rm weak} \ \mbox{as}\ k \to \infty,
\]
where 	
\[
(\vr, \vm, \mathcal{E}) \in \mathcal{U}[\vr_0, \vm_0, \mathcal{E}_0].
\]
	
	\end{Proposition}

The sequential stability yields, in particular, the closedness
of the graph of the multivalued solution mapping that is crucial for
Borel measurability of the latter with respect to the data discussed in the next section.

\subsubsection{Measurability of the solution set in $\mathcal{T}_{\rm weak}$}

We recall that a set--valued mapping $\mathcal{U} :\mathcal{D}
\to 2^\mathcal{T}$ ranging in a family of subsets of a topological space $\mathcal{T}$ is (weakly)
measurable if the set
\[
\left\{ d \in \mathcal{D} \ \Big| \ \mathcal{U}(d) \cap B \ne \emptyset \right\}
\]
are measurable for any open set $B \subset \mathcal{T}$. It is worth noting that
$\mathcal{U}$ is measurable if and only if ${\rm cl}[\mathcal{U}]$ is measurable. If $\mathcal{D}$ is a topological space, we may define Borel
measurability requiring the set to be Borel measurable.

If $X$ is a separable complete metric space and the sets $\mathcal{U}$
are compact, (Borel) measurability is equivalent to (Borel) measurability
of the mapping
\[
\mathcal{U}: \mathcal{D} \to {\rm comp}[\mathcal{T}],
\]
where ${\rm comp}[\mathcal{T}]$ is the metric space of all compact subsets of $\mathcal{T}$ endowed with the Hausdorff topology.

As shown in \cite[Section 3.1]{BreFeiHof19}, the solution mapping
\begin{equation} \label{s9}
\mathcal{U}: \mathcal{D} \subset W^{-\ell,2}(\Omega) \times
W^{-\ell,2}(\Omega; R^d) \times R \to
{\rm comp}[ \mathcal{T}_{\rm weak}]
\end{equation}
is Borel measurable.

\subsubsection{Measurability of the solution set in $\mathcal{T}_{\rm strong}$}

As the solution mapping is Borel measurable with values in compact subsets of $\mathcal{T}_{\rm weak}$, it admits the so--called Castaign representation -- a countable family of Borel measurable selections
\[
U_i : \mathcal{D} \to \mathcal{T}_{\rm weak} ,\
U_i (d) \in \mathcal{U}_i [d] \ \mbox{for any}\ d \in \mathcal{D},\
i=1,2,\dots,
\]
such that
\[
{\rm cl}_{\mathcal{T}_{\rm weak}}[ (U_i(d))_{i=1}^\infty ] =
\mathcal{U}[d] \ \mbox{for any}\ d \in \mathcal{D}.
\]

Since the sets $\mathcal{U}$ are convex, we may extend the family
$U_i$ by considering convex combinations
\[
\sum_{j=1}^m \lambda_j U_i ,\ \lambda_j \in Q^+,\ \sum_{j=1}^m \lambda_j = 1
\]
with rational coefficients. As the closures of convex sets in the weak and strong topology coincide, we infer that the augmented system is a Castaign representation with respect to the topology $\mathcal{T}_{\rm strong}$.
Thus
\begin{equation} \label{s10}
\mathcal{U}: \mathcal{D} \subset W^{-\ell,2}(\Omega) \times
W^{-\ell,2}(\Omega; R^d) \times R \to
{\rm closed}[ \mathcal{T}_{\rm strong}]
\end{equation}
is Borel measurable.

Finally, by Hess' measurability theorem \cite{Hess}, the measurability stated in \eqref{s10} is equivalent to the Borel measurability of the set valued mapping
\[
\mathcal{U}: \mathcal{D} \to {\rm closed}[{\mathcal{T}}_{\rm strong}],
\]
where ${\rm closed}[\mathcal{T}_{\rm strong}]$ is the set of all closed subsets
of $\mathcal{T}_{\rm strong}$ endowed with the (metrizable) Wijsman topology:
\begin{equation} \label{s11}
\mathcal{A}_n \toW \mathcal{A} \ \Leftrightarrow\
{\rm dist}_{\mathcal{T}_{\rm strong}}[y, \mathcal{A}_n] \to {\rm dist}_{\mathcal{T}_{\rm strong}}[y, \mathcal{A}] \ \mbox{for any}\ y \in \mathcal{T}_{\rm strong}.
\end{equation}

\section{Regularity of dissipative solutions}
\label{d}

We are ready to state and prove our main results concerning regularity of dissipative solutions. They include, in particular:
\begin{itemize}
	\item The existence of dissipative solutions with arbitrarily small Reynolds stress (energy defect).
	\item Regularity of maximal dissipative solutions, namely any dissipative solution minimal with respect to $\prec_{\rm loc}$ is an admissible
	weak solution.
\end{itemize}

\subsection{Dissipative solutions with small energy defect}
\label{sed}

We show the existence of dissipative solutions with arbitrarily small energy defect
\[
{D}_{\mathcal{E}}(\tau) = \mathcal{E}(\tau+) - \intO{ E(\vr, \vm)(\tau, \cdot) }.
\]
In view of the compatibility relation \eqref{w9}, smallness of $D_{\mathcal{E}}$ entails smallness of the norm of the Reynolds stress $\mathcal{R}$.

The result follows from two crucial properties of dissipative solutions:
\begin{itemize}
\item
{\bf Initial time regularity.} Suppose that
\begin{equation} \label{d1}
\mathcal{E}(\tau+) = \intO{ E(\vr, \vm) (\tau)  }
\ \Leftrightarrow \ D_{\mathcal{E}}(\tau+) = 0
\ \mbox{for some}\ \tau \geq 0. 	
\end{equation}
Then for any $\delta > 0$ there exists $T = T(\delta) > 0$ such that
\begin{equation} \label{d2}
0 \leq {D}_{\mathcal{E}}(t) \leq \delta \ \mbox{for any}\ \tau \leq t < \tau + T(\delta).
\end{equation}	
Indeed, as the energy is weakly l.s.c. and the total energy $\mathcal{E}$
non--increasing, we have
\begin{align}
0 \leq \limsup_{t \to \tau +} {D}_{\mathcal{E}} (t) &\leq
\limsup_{t \to \tau +} \mathcal{E}(t+) - \liminf_{t \to \tau}
\intO{ E(\vr, \vm)(t, \cdot) }\br &\leq \mathcal{E}( \tau +) - \intO{ E(\vr, \vm)(\tau, \cdot) } = 0.
\nonumber
\end{align}

\item {\bf Concatenation.}

Suppose $(\vr^1, \vm^1, \mathcal{E}^1) \in \mathcal{U}[\vr_0, \vm_0, \mathcal{E}_0]$
and $(\vr^2, \vm^2, {\mathcal{E}^2}) \in
\mathcal{U}[\vr(T, \cdot), \vm(T, \cdot), \mathcal{E}^2_T]$,
\[
\intO{ E(\vr, \vm)(T, \cdot) } \leq \mathcal{E}^2_T \leq \mathcal{E}^1(T).
\]
Then
\[
(\vr, \vm, \mathcal{E})(t, \cdot) = \left\{ \begin{array}{l}(\vr^1, \vm^1 , \mathcal{E}^1)(t, \cdot)
\ \mbox{for}\ 0 \leq t \leq T, \\ \\
(\vr^2, \vm^2 , \mathcal{E}^2)(t - T, \cdot)
\ \mbox{for}\ t > T,
\end{array} \right.
\]
belongs to the the solution set $\mathcal{U}[\vr_0, \vm_0, \mathcal{E}_0]$.
\end{itemize}	

Our main result stated below is based on an application of Zorn's lemma equivalent to the axiom of choice.

\begin{mdframed}[style=MyFrame]

\begin{Theorem}[{\bf Dissipative solutions with small defect}] \label{dT1}
Let $\delta > 0$ and the initial data
\[
(\vr_0, \vm_0, \mathcal{E}_0) \in \mathcal{D},\ \mathcal{E}_0 = \intO{ E(\vr_0, \vm_0) },
\]
be given.

Then there exists a dissipative solution
\[
(\vr, \vm, \mathcal{E} ) \in \mathcal{U}[\vr_0, \vm_0, \mathcal{E}_0]
\]
such that
\begin{equation} \label{d3}
	D_{\mathcal{E}} (t) \leq \delta \ \mbox{for all}\ t \geq 0.	
\end{equation}
\end{Theorem}

\end{mdframed}

\begin{proof}
	
For each global solution
\[
(\vr, \vm, \mathcal{E}) \in \mathcal{U}[\vr_0, \vm_0, \mathcal{E}_0]
\]
we define a ``stopping'' time
\[
T(\delta)[\vr, \vm, \mathcal{E}] = \sup \left\{ \tau \geq 0 \ \Big|\
D_{\mathcal{E}} (t) \leq \delta \ \mbox{for any}\ 0 \leq t < \tau \right\}.
\]
We know from \eqref{d2} that $T(\delta) > 0$. Moreover,
\begin{equation} \label{d4}
D_{\mathcal{E}}(t) \leq \delta \ \mbox{for any}\ 0 \leq t < T(\delta).	
\end{equation}

Next, we introduce the order relation
\begin{align}
(\vr^1, \vm^1, \mathcal{E}^1) &\prec \hspace{-0.2cm} \prec (\vr^2, \vm^2, \mathcal{E}^2) \br &\Leftrightarrow \br
\mbox{Either}\ (\vr^1, \vm^1, \mathcal{E}^1) &= (\vr^2, \vm^2, \mathcal{E}^2),
\br
\mbox{or}\
T(\delta)[\vr^1, \vm^1, \mathcal{E}^1] &< \infty,\ T(\delta)[\vr^1, \vm^1, \mathcal{E}^1]< T(\delta) [\vr^2, \vm^2, \mathcal{E}^2], \br
&\mbox{and}\ (\vr^1, \vm^1, \mathcal{E}^1)(t) =
(\vr^2, \vm^2, \mathcal{E}^2)(t)  \ \mbox{for any}\ 0 \leq t < T(\delta)[\vr^1, \vm^1, \mathcal{E}^1].
\label{d5}
\end{align}
It is a routine matter to check that $\mathcal{U}[\vr_0, \vm_0, \mathcal{E}_0]$ endowed with the relation $\prec \hspace{-0.2cm} \prec$ is a partially ordered set
eligible for the application of Zorn's lemma.

Let
\[
(\vr^1, \vm^1, \mathcal{E}^1) \prec \hspace{-0.2cm} \prec
(\vr^2, \vm^2, \mathcal{E}^2) \prec \hspace{-0.2cm} \prec \dots
\prec \hspace{-0.2cm} \prec(\vr^n, \vm^n, \mathcal{E}^n),\
(\vr^i, \vm^i, \mathcal{E}^i) \in \mathcal{U}[\vr_0, \vm_0, \mathcal{E}_0]
\ \mbox{for}\ i=1,2, \dots
\]
be an ordered chain. Without loss of generality, we may assume $(\vr^i, \vm^i, \mathcal{E}^i) \ne
(\vr^j, \vm^j, \mathcal{E}^j)$ for $i \ne j$.
As the solution set is compact in $\mathcal{T}_{\rm weak}$, the
sequence $(\vr^n, \vm^n, \mathcal{E}^n)_{n=1}^\infty$ contains a converging subsequence (not relabeled)
\[
(\vr^n, \vm^n, \mathcal{E}^n) \to (\tvr, \tvm, \widetilde{\mathcal{E}}) \in \mathcal{U}[\vr_0, \vm_0, \mathcal{E}_0] \ \mbox{as}\ n \to \infty
\ \mbox{in} \ \mathcal{T}_{\rm weak}.
\]
In accordance with \eqref{d5},
\[
(\tvr, \tvm, \widetilde{ \mathcal{E}})(t) = (\vr^n, \vm^n, \mathcal{E}^n)(t) \ \mbox{for}\
t < T(\delta)[\vr^n, \vm^n, \mathcal{E}^n] \ \mbox{for any}\ n = 1,2,\dots.
\]

Now, either $T(\delta)(\tvr, \tvm, \widetilde{\mathcal{E}}) = \infty$, meaning
$(\tvr, \tvm, \widetilde{ \mathcal{E}})$ is an upper bound for the chain, or
$T(\delta)[\tvr, \tvm, \widetilde{\mathcal{E}}]= \widetilde{T} < \infty$. In the latter case, we consider a concatenation
\[
(\vr, \vm, \mathcal{E})(t, \cdot) = \left\{ \begin{array}{l}(\tvr, \tvm , \widetilde{\mathcal{E}})(t, \cdot)
	\ \mbox{for}\ 0 \leq t \leq \widetilde{T}, \\ \\
	(\hat{\vr}, \hat{\vm} , \hat{\mathcal{E}})(t - \widetilde{T}, \cdot)
	\ \mbox{for}\ t > \widetilde{T},
\end{array} \right.
\]
where
\[
(\hat{\vr}, \hat{\vm}, \hat{\mathcal{E}}) \in \mathcal{U}\Big(\tvr(\widetilde{T}, \cdot), \tvm (\widetilde{T}, \cdot),
\intO{ E(\tvr(\widetilde{T}, \cdot), \tvm (\widetilde{T}, \cdot) )} \Big).
\]
Applying \eqref{d1}, \eqref{d2} to $(\hat{\vr}, \hat{\vm}, \hat{\mathcal{E}})$ we can see $({\vr}, {\vm}, {\mathcal{E}})$ is an upper bound for the
chain. Thus in both cases, the ordered chain admits an upper bound in $\mathcal{U}[\vr_0, \vm_0, \mathcal{E}_0]$.
	
By Zorn's lemma, there exists a maximal solution
$(\tvr, \tvm, \widetilde{\mathcal{E}}) \in
\mathcal{U}[\vr_0, \vm_0, \mathcal{E}_0]$ with respect to the relation
$\prec \hspace{-0.2cm} \prec$. We claim
\[
T(\delta) [\tvr, \tvm, \widetilde{\mathcal{E}}] = \infty,
\]
which completes the proof. Indeed, if
\[	
T(\delta) [\tvr, \tvm, \widetilde{\mathcal{E}}] = \widetilde{T} < \infty,
\]
we could repeat the above argument concatenating the solution
$
(\tvr, \tvm, \widetilde{\mathcal{E}})
$
with another solution starting from the initial data
\[
(\hat{\vr}, \hat{\vm}, \hat{\mathcal{E}}) \in \mathcal{U}\Big(\tvr(\widetilde{T}, \cdot), \tvm (\widetilde{T}, \cdot),
\intO{ E(\tvr(\widetilde{T}, \cdot), \tvm (\widetilde{T}, \cdot) )} \Big),
\]
\[
(\vr, \vm, \mathcal{E})(t, \cdot) = \left\{ \begin{array}{l}(\tvr, \tvm , \widetilde{\mathcal{E}})(t, \cdot)
	\ \mbox{for}\ 0 \leq t \leq \widetilde{T}, \\ \\
	(\hat{\vr}, \hat{\vm} , \hat{\mathcal{E}})(t - \widetilde{T}, \cdot)
	\ \mbox{for}\ t > \widetilde{T}.
\end{array} \right.
\]
Obviously, the new solution $(\vr, \vm, \mathcal{E})$ satisfies
\[
(\tvr, \tvm, \widetilde{\mathcal{E}})
\pprec (\vr, \vm, \mathcal{E}),
\]
and, by virtue of the initial regularity property,
\[
T(\delta) [{\vr}, \vm, \mathcal{E}] \geq \widetilde{T} +
T(\delta)
[\hat{\vr}, \hat{\vm}, \hat{\mathcal{E}}] > \widetilde{T}.
\]	
Consequently, $(\tvr, \tvm, \widetilde{\mathcal{E}})$ is not maximal --
a contradiction.

	\end{proof}

Applying the same approach we can show the existence of a dissipative solution with the energy defect dominated by an arbitrary positive function, in particular, the defect may vanish for $t \to \infty$.

\begin{mdframed}[style=MyFrame]

\begin{Corollary} \label{dC1}
Let a function $\delta \in C[0, \infty)$, $\delta(t) > 0$ for any $t \geq 0$, and the
initial data 	
\[
(\vr_0, \vm_0, \mathcal{E}_0) \in \mathcal{D},\ \mathcal{E}_0 = \intO{ E(\vr_0, \vm_0) },
\]
be given.

Then there exists a dissipative solution
\[
(\vr, \vm, \mathcal{E} ) \in \mathcal{U}[\vr_0, \vm_0, \mathcal{E}_0]
\]
such that
\[
D_{\mathcal{E}}(t) \leq \delta(t) \ \mbox{for all}\ t > 0.	
\]
	
\end{Corollary}

\end{mdframed}

\begin{proof}
Apply the arguments of the proof of Theorem \ref{dT1} to obtain
a solution $(\vr, \vm, \mathcal{E})$, with the energy defect
\[
D_{\mathcal{E}}(t) \leq \min_{\tau \in [n, n+1]} \delta (\tau) \ \mbox{for}\
t \in [n, n+1].
\]
	
\end{proof}
	
\subsection{Maximal dissipation principle}

We recall the order relation introduced in \eqref{Dafer}, namely
\begin{align} \nonumber
	(\vr^1, \vm^1, \mathcal{E}^1) &\prec_{\rm loc}
	(\vr^2, \vm^2, \mathcal{E}^2)\br &\Leftrightarrow \br
	\mbox{there exists}\ T &\geq 0 \ \mbox{such that}\
	(\vr^1, \vm^1, \mathcal{E}^1)(\tau, \cdot) =
	(\vr^2, \vm^2, \mathcal{E}^2)(\tau, \cdot) \ \mbox{for any}\ \tau \in [0,T],\br
	\mathcal{E}^1(\tau) &< \mathcal{E}^2(\tau) \ \mbox{for}\
	\tau \in (T, T+ \delta) \ \mbox{for some}\ \delta > 0.
	\nonumber
\end{align}	
Similarly to the relation $\prec \hspace{-0.2cm} \prec$ introduced in the proof of Theorem \ref{dT1},  the order relation
$\prec_{\rm loc}$ augmented formally by the identity relation represents a partial ordering of $\mathcal{U}[\vr_0, \vm_0, \mathcal{E}_0]$.
Recall that a dissipative solution is maximal (dissipative) if it is minimal with respect to $\prec_{\rm loc}$.

\subsubsection{Regularity of maximal dissipative solutions}

Our next goal is to show that any maximal dissipative solution in the sense of Definition \ref{wD4} is necessarily an admissible (with non--increasing total energy) weak solution.
The result can be seen as a rigorous verification of DiPerna's conjecture \cite[Section 6]{DiP2} in the context of the barotropic Euler system.

\begin{mdframed}[style=MyFrame]

\begin{Theorem}[{\bf Regularity of maximal dissipative solutions}] \label{dT2}
	
Suppose \[
(\vr, \vm, \mathcal{E}) \in \mathcal{U}[\vr_0, \vm_0, \mathcal{E}_0]
\]
is a maximal
dissipative solution.

Then the energy defect vanishes,
\[
D_{\mathcal{E}}(t) = \mathcal{E}(t+) - \intO{ E(\vr, \vm) (t) }
 = 0
\]
for any $t \geq 0$.
In particular, $(\vr, \vm)$ is a weak solution of the Euler system (with a non--increasing total energy profile).
\end{Theorem}

\end{mdframed}

\begin{proof}
	
Suppose
\[
D_{\mathcal{E}}(T) = \mathcal{E}(T+) - \intO{ E(\vr, \vm) (T) } = \ep > 0
\]	
for some $T \geq 0$. As $\mathcal{E}$ is non--increasing, there exists $\delta > 0$ such that
\begin{equation} \label{d7}
\mathcal{E}(t) > \mathcal{E}(T+) - \frac{\ep}{2} \ \mbox{for any}\ t \in (T, T + \delta).
\end{equation} 	

Next, by concatenation, we may construct a new solution of the same problem such that
\[
(\tvr, \tvm, \widetilde{\mathcal{E}})(t, \cdot) = \left\{ \begin{array}{l}(\vr, \vm , {\mathcal{E}})(t, \cdot)
	\ \mbox{for}\ 0 \leq t \leq T, \\ \\
	(\hat{\vr}, \hat{\vm} , \hat{\mathcal{E}})(t - \widetilde{T}, \cdot)
	\ \mbox{for}\ t > {T},
\end{array} \right.
\]
where
\[
(\hat{\vr}, \hat{\vm}, \hat{\mathcal{E}}) \in \mathcal{U}\Big(\vr({T}, \cdot), \vm ({T}, \cdot),
\intO{ E(\vr({T}, \cdot), \vm ({T}, \cdot) )} \Big).
\]
Obviously,
\[
(\tvr, \tvm, \widetilde{\mathcal{E}})(t, \cdot) =
(\vr, \vm, \mathcal{E})(t, \cdot) \ \mbox{for all}\ t \in [0,T],
\]
while, by virtue of \eqref{d7},
\[
\widetilde{\mathcal{E}} (t) \leq \intO{ E(\vr({T}, \cdot), \vm ({T}, \cdot)) }
= \mathcal{E}(T+) - \ep < \mathcal{E}(t)
\]
for any $t \in (T, T+ \delta)$.
We conclude that $(\vr, \vm, \mathcal{E})$ is not minimal.
	
	\end{proof}

As we shall see in Section \ref{ss}, the existence of admissible dissipative solutions, meaning solutions minimal with respect
to the order $\prec$ is guaranteed for any finite energy data. Unfortunately, however, the existence of
maximal dissipative solutions, meaning minimal with respect
to the order $\prec_{\rm loc}$ remains an open problem.

\section{Semigroup (semiflow) selection}
\label{ss}

Before discussing any selection process, we define two operations on the
trajectory space $\mathcal{T}_{\rm weak}$:
\begin{itemize}
	\item {\bf Time shift.} For any $T \geq 0$ and any
	$(\vr, \vm, \mathcal{E}) \in \mathcal{T}_{\rm weak}$, we define
	\[
	\mathcal{S}_T[ 	\vr, \vm, \mathcal{E} ](t, \cdot) =
	(\vr (t + T), \vm (t +T), \mathcal{E}(t + T) ),\ t \geq 0.
	\]
	
	\item{\bf Concatenation.} 	
	For
	\[
	(\vr^i, \vm^i, \mathcal{E}^i) \in \mathcal{T}_{\rm weak}, \ i=1,2,\ T > 0.
	\]
	we define
	\[
	(\vr^1, \vm^1, \mathcal{E}^1) \cup_T
	(\vr^2, \vm^2, \mathcal{E}^2) = (\vr, \vm, \mathcal{E}),
	\]
	\[
	(\vr, \vm, \mathcal{E}) = \left\{
	\begin{array}{l}
		(\vr^1, \vm^1, \mathcal{E}^1) (t, \cdot) \ \mbox{for}\ 0 \leq t \leq T, \\ \\
		(\vr^2, \vm^2, \mathcal{E}^2) (t - T, \cdot) \ \mbox{for}\ t > T	
	\end{array} \right.
	\]
	Note that $(\vr^1, \vm^1, \mathcal{E}^1) \cup_T (\vr^2, \vm^2, \mathcal{E}^2)$
	belongs to $\mathcal{T}_{\rm weak}$ provided
	\[
	(\vr^2(0, \cdot), \vm^2(0, \cdot)) = (\vr^1(T, \cdot), \vm^1(T, \cdot)).
	\]
	
\end{itemize}

\subsection{Properties of the solution sets}

We recall the basic properties of the solution sets $\mathcal{U}$ established in Section \ref{s}.

\begin{itemize}
	\item {\bf [A1]} For any data $(\vr_0, \vm_0, \mathcal{E}_0) \in \mathcal{D}$, the solution set
\[
\mathcal{U}	[\vr_0, \vm_0, \mathcal{E}_0] \subset
\mathcal{T}_{\rm weak} \cap \mathcal{T}_{\rm strong}
\]
is a non--empty convex set, compact in $\mathcal{T}_{\rm weak}$ and closed bounded in $\mathcal{T}_{\rm strong}$.

\item {\bf [A2]} If
\[
(\vr, \vm, \mathcal{E}) \in \mathcal{U}[\vr_0, \vm_0, \mathcal{E}_0],
\]
than the time shift satisfies
\[
\mathcal{S}_T [\vr, \vm, \mathcal{E}]
\in \mathcal{U}[ \vr(T, \cdot),
\vm (T, \cdot), \mathcal{E}(T)    ]
\]
for any $T \geq 0$.

\item {\bf [A3]} If
\[
(\vr^1, \vm^1, \mathcal{E}^1) \in \mathcal{U}[\vr_0, \vm_0, \mathcal{E}_0]
\]
and
\[
(\vr^2, \vm^2, \mathcal{E}^2) \in \mathcal{U}[\vr^1(T, \cdot), \vm^1(T, \cdot), \mathcal{E}^2_0], \ \mbox{where}\
\intO{ E(\vr^1, \vm^1)(T, \cdot) } \leq \mathcal{E}^2_0 \leq
\mathcal{E}^1(T),
\]
then
\[
(\vr^1, \vm^1, \mathcal{E}^1) \cup_T (\vr^2, \vm^2, \mathcal{E}^2) \in
\mathcal{U}[\vr_0, \vm_0, \mathcal{E}_0].
\]

\item {\bf [A4]}
The set valued mappings
\[
(\vr_0, \vm_0, \mathcal{E}_0) \in \mathcal{D}
\mapsto \mathcal{U}[\vr_0, \vm_0, \mathcal{E}_0] \in {\rm comp}[\mathcal{T}_{\rm weak}]
\]
and
\[
(\vr_0, \vm_0, \mathcal{E}_0) \in \mathcal{D}
\mapsto \mathcal{U}[\vr_0, \vm_0, \mathcal{E}_0] \in {\rm closed}[ \mathcal{T}_{\rm strong}]
\]
are Borel measurable with respect to the Hausdorff and the Wijsman topology, respectively.
	
\end{itemize}

\subsection{Selection process}

We consider a functional of the form
\[
\mathcal{F}(\vr, \vm, \mathcal{E}) = \int_0^\infty \exp(-t)
F(\vr(t, \cdot), \vm(t, \cdot), \mathcal{E}(t, \cdot)) \dt,
\]
where
\[
F: \mathcal{D} \to R
\]
is a convex l.s.c. function.

For each solution set
\[
\mathcal{U}[\vr_0, \vm_0, \mathcal{E}_0],
\]
let
\begin{align}
\mathcal{U}_{\mathcal{F}}[\vr_0, \vm_0, \mathcal{E}_0] &=
\left\{ (\tvr, \tvm, \widetilde{E}) \in
\mathcal{U}[\vr_0, \vm_0, \mathcal{E}_0]
\ \Big| \ \mathcal{F} (\tvr, \tvm, \widetilde{\mathcal{E}}) \leq \mathcal{F} (\vr, \vm, \mathcal{E}) \right. \br  &\mbox{for any}\  (\vr, \vm, \mathcal{E}) \in
\mathcal{U}[\vr_0, \vm_0, \mathcal{E}_0]
\Big\}
\nonumber
\end{align}
be the set of minimizers of $\mathcal{F}$ on $\mathcal{U}[\vr_0, \vm_0, \mathcal{E}_0]$.

It is easy to see that the restricted solution set
$\mathcal{U}_{\mathcal{F}}[\vr_0, \vm_0, \mathcal{E}_0]$
is again non--empty and convex. Next, we show that
the properties {\bf [A2]}, {\bf [A3]} are satisfied. As a matter of fact, the arguments are almost identical to those of  \cite[Section 5]{BreFeiHof19}.
As the present setting is slightly more general, we reproduce the proof for reader's convenience.

\begin{Lemma} \label{SL1}
\begin{itemize}
	\item {\bf Shift property.}
\[
(\vr, \vm, \mathcal{E}) \in \mathcal{U}_{\mathcal{F}}[\vr_0, \vm_0, \mathcal{E}_0] \ \Rightarrow\ \mathcal{S}_T (\vr, \vm, \mathcal{E})
\in \mathcal{U}_{\mathcal{F}}[ \vr(T, \cdot),
\vm (T, \cdot), \mathcal{E}(T)]
\]	
	\item {\bf Concatenation property.}
	
	\[
	(\vr^1, \vm^1, \mathcal{E}^1) \in \mathcal{U}_{\mathcal{F}}[\vr_0, \vm_0, \mathcal{E}_0]
	\]
	and
	\[
	(\vr^2, \vm^2, \mathcal{E}^2) \in \mathcal{U}_{\mathcal{F}}[\vr^1(T, \cdot), \vm^1(T, \cdot), \mathcal{E}^2_0], \ \mbox{where}\
	\intO{ E(\vr^1, \vm^1)(T, \cdot) } \leq \mathcal{E}^2_0 \leq
	\mathcal{E}^1(T),
	\]
	\centerline{$\Rightarrow$}
	\[
	(\vr^1, \vm^1, \mathcal{E}^1) \cup_T (\vr^2, \vm^2, \mathcal{E}^2) \in
	\mathcal{U}_{\mathcal{F}}[\vr_0, \vm_0, \mathcal{E}_0].
	\]
	
\end{itemize}		
	
	\end{Lemma}
	
\begin{proof}
	
The proof is basically the same as in Cardona, Kapintanski \cite{CorKap} or \cite{BreFeiHof19}. We present is here for reader's convenience.
	
\noindent	
{\bf Shift property.}

First, by the continuation property,
\[
(\vr^T, \vm^T, \mathcal{E}^T)\in \mathcal{U}_{\mathcal{F}}[\vr(T), \vm(T), \mathcal{E}(T)] \]
\[
\Rightarrow \ (\vr, \vm, \mathcal{E}) \cup_T[\vr^T, \vm^T, \mathcal{E}^T]\in \mathcal{U}[\vr_0, \vm_0, E_0].
\]
Now
\begin{align}
\mathcal{F}(S_T (\vr, \vm, \mathcal{E})) &= \int_0^\infty \exp(-t)
F((\vr, \vm, \mathcal{E})(t+T)) \dt  \br
&= \exp(T) \left( \mathcal{F}((\vr, \vm, \mathcal{E})) -
\int_0^T \exp(-t) F((\vr, \vm, \mathcal{E})(t)) \dt \right) \br
&\leq \exp(T) \left( \mathcal{F}((\vr, \vm, \mathcal{E}) \cup_T
(\vr^T, \vm^T, \mathcal{E}^T)) -
\int_0^T \exp(-t) F((\vr, \vm, \mathcal{E})(t)) \dt \right) \br
&= \exp(T) \int_T^\infty \exp(-t) F((\vr^T, \vm^T, \mathcal{E}^T)(t - T)) \dt
= \mathcal{F}(\vr^T, \vm^T, \mathcal{E}^T ).
\nonumber
\end{align}

\noindent
{\bf Concatenation property.}	

Compute
\begin{align}
&\mathcal{F}\Big( (\vr^1, \vm^1, \mathcal{E}^1)\cup_T (\vr^2, \vm^2, \mathcal{E}^2) \Big) \br
&\quad=\int_0^T \exp (-t) F((\vr^1, \vm^1, \mathcal{E}^1)(t))\dt+\int_T^\infty \exp(-t) F((\vr^2, \vm^2, \mathcal{E}^2)(t-T))\dt\br
&\quad=\int_0^T \exp(-t) F((\vr^1, \vm^1, \mathcal{E}^1)(t))\dt+ \exp(-T) \mathcal{F}(\vr^2, \vm^2, \mathcal{E}^2)\br
&\quad \leq \int_0^T \exp(-t) F((\vr^1, \vm^1, \mathcal{E}^1)(t))\dt+ \exp(-T) \mathcal{F}(S_T(\vr^1, \vm^1, \mathcal{E}^1))\br
&\quad= \mathcal{F}(\vr^1, \vm^1, \mathcal{E}^1).	
	\nonumber
\end{align}	
	
\end{proof}

\subsection{First selection criterion}
\label{F}

The first selection criterion consists in replacing the set
\[
\mathcal{U}[\vr_0, \vm_0, \mathcal{E}_0]
\]
by the set of minimizers of the functional
\begin{equation} \label{F1}
\mathcal{F}_1(\vr, \vm, \mathcal{E}) = \int_0^\infty \exp(-t) \mathcal{E}(t) \dt
\end{equation}
considering the topology of the space $\mathcal{T}_{\rm weak}$. Note that
$\mathcal{F}_1(\vr, \vm, \mathcal{E})$ is a bounded (continuous) linear form
on both $\mathcal{T}_{\rm weak}$ and $\mathcal{T}_{\rm strong}$.

Thus we set
\begin{align}
	\mathcal{U}_{\mathcal{F}_1}[\vr_0, \vm_0, \mathcal{E}_0]
&= \left\{ (\tvr, \tvm, \tvmE) \in \mathcal{U}[\vr_0, \vm_0, \mathcal{E}_0]  \ \Big| \
\int_0^\infty \exp(-t) \mathcal{E}(t) \dt \leq
\int_0^\infty \exp(-t) \widetilde{\mathcal{E}}(t) \dt \right. \br
&\left.
\quad \quad \mbox{for all}\ (\tvr, \tvm, \widetilde{\mathcal{E}}) \in
	\mathcal{U}[\vr_0, \vm_0, \mathcal{E}_0] 	\right\} \br
	&\equiv {\arg \min}_{\mathcal{U}[\vr_0, \vm_0, \mathcal{E}_0]} \mathcal{F}_1(\vr, \vm, \mathcal{E}).
\label{F2}
\end{align}
The resulting solution sets are again non--empty, convex and compact in $\mathcal{T}_{\rm weak}$. Moreover, as shown in \cite{BreFeiHof19}, the mapping
\[
\mathcal{A} \subset \mathcal{T}_{\rm weak} \mapsto
\left\{
(\vr, \vm, \mathcal{E}) \in \mathcal{A} \ \Big| \
(\vr, \vm, \mathcal{E}) = \arg \min_{\mathcal{A}} \mathcal{F}_1(\vr, \vm, \mathcal{E}) \subset \mathcal{A}
\right\}
\]
is continuous in ${\rm comp}[\mathcal{T}_{\rm weak}]$
endowed with the Hausdorff topology.

It is worth noting that we minimize
\[
\int_0^\infty \exp(-t) \mathcal{E}(t) \dt =
\int_0^\infty \exp(-t) \left[ {\mathcal{E}}(t+) - \intO{ E(\vr, \vm) (t. \cdot)}\right] +
\intO{ E(\vr, \vm)(t, \cdot)}
\dt,
\]
where
\[
D_{\mathcal{U}}(t)
{\mathcal{E}}(t+) - \intO{ E(\vr, \vm)(t ,\cdot)}  \geq 0
\]
is the energy defect, and
\[
\intO{ E(\vr, \vm)(t, \cdot) }
\]
the ``real'' energy of the fluid flow. Moreover, as observed in \cite[Section 5]{BreFeiHof19}, all selected solutions are minimal
with respect to the relation $\prec$, meaning admissible.

\subsection{Second selection criterion}
\label{S}

The next step is to minimize
\begin{equation} \label{FF1}
\mathcal{F}_2(\vr, \vm, \mathcal{E}) =
\int_0^\infty \exp(-t) \left[ \| \vr \|^q_{L^q(\Omega)} +
\| \vm \|^q_{L^q(\Omega; R^d)} + |\mathcal{E}|^q \right] \dt
\end{equation}
on $\mathcal{U}_{\mathcal{F}_1}[\vr_0, \vm_0, \mathcal{E}_0]$,
where $1 < q \leq \frac{2 \gamma}{\gamma + 1}$.
As the functional is strictly convex, there is a unique minimum attained on  any closed convex subset of $\mathcal{T}_{\rm strong}$. The selected
triple $(\vr, \vm, \mathcal{E})$ represents a semigroup selection for the compressible Euler system.

In order to show Borel measurability of the data-to-solution mapping, it is enough to show continuity of the mapping
\[
\mathcal{A} \in {\rm closed}[\mathcal{T}_{\rm strong}] \mapsto \arg \min_{\mathcal{A}} F_2 (\vr, \vm, \mathcal{E})
\in \mathcal{T}_{\rm strong}
\]
for $\mathcal{A}$ a closed convex subset belonging ${\rm closed}[\mathcal{T}_{\rm strong}]$ endowed with the Wijsman topology.

The space $\mathcal{T}_{\rm strong}$ is a reflexive, separable, uniformly convex Banach space with a uniformly convex dual. Consequently, the convergence in the Wijsman topology is equivalent to convergence in the Mosco topology $\mathcal{M}$, see \cite[Sonntag-Attouch Theorem in Section 4]{Beer1}.
We recall that
\[
\mathcal{A}_n \toMo \mathcal{A}
\]
if the two following properties hold:
\begin{itemize}
	\item For any $(\vr, \vm, \mathcal{E}) \in \mathcal{A}$, there exists
	$(\vr_n, \vm_n, \mathcal{E}_n) \in \mathcal{A}_n$ such that
\[
(\vr_n, \vm_n, \mathcal{E}_n) \to (\vr, \vm, \mathcal{E}) \ \mbox{in}\ \mathcal{T}_{\rm strong};
\]	
\item If
\[
(\vr_{n_k}, \vm_{n_k}, \mathcal{E}_{n_k}) \in \mathcal{A}_{n_k} \to (\vr, \vm, \mathcal{E})
\ \mbox{weakly in}\ \mathcal{T}_{\rm strong}
\]
then
\[
(\vr, \vm, \mathcal{E}) \in \mathcal{A}.
\]
\end{itemize}	
	
These conditions guarantee
\[
\arg \min_{(\vr, \vm, \mathcal{E}) \in \mathcal{A}_n} \mathcal{F}_2 (\vr, \vm, \mathcal{E}) \to \arg \min_{(\vr, \vm, \mathcal{E}) \in \mathcal{A}} \mathcal{F}_2 (\vr, \vm, \mathcal{E})
\]	
whenever
\[
\mathcal{A}_n \toMo \mathcal{A},\ \mathcal{A}_n, \mathcal{A}
\ \mbox{closed convex subsets of}\ \mathcal{T}_{\rm strong}.
\]

The second selection is therefore Borel measurable mapping from
$\mathcal{D}$ to $\mathcal{T}_{\rm strong}$.

\subsubsection{Minimizing the momentum}

Alternatively, we can minimize only the momentum in the second step:
\begin{equation} \label{FF2}
\mathcal{F}_2(\vr, \vm, \mathcal{E}) = \int_0^\infty \exp(-t)
\|\vm (t, \cdot) \|^q_{L^q(\Omega; R^d)} \ \dt
\end{equation}
considered as a continuous convex functional on $\mathcal{T}_{\rm strong}$.

The only difference in comparison with the previous section is the fact that
\[
\arg \min_{\mathcal{A}} \mathcal{F}_2
\]
is now a set (not a single point) in $\mathcal{T}_{\rm weak}$.
Accordingly, the Borel measurability will follow, as soon as we show that the multivalued mapping
\[
\arg \min_{\mathcal{A}} \mathcal{F}_2 : {\rm convex}[\mathcal{T}_{\rm strong}] \to 2^{\mathcal{T}_{\rm weak}}
\]
has a closed graph, where the space
\[
{\rm convex}[\mathcal{T}_{\rm strong}]
\]
is endowed with the metrizable Mosco topology.

Suppose
\[
\mathcal{A}_n \toMo \mathcal{A} \ \mbox{and}\
(\vr_n, \vm_n, \mathcal{E}) \in \arg \min_{\mathcal{A}_n} \mathcal{F}_2
\to (\vr, \vm, \mathcal{E}) \in \mathcal{T}_{\rm weak}.
\]
On the one hand,
it follows from convexity and continuity of $\mathcal{F}_2$ that
\[
\mathcal{F}_2 (\vr, \vm, \mathcal{E}) \leq \liminf_{n \to \infty} \mathcal{F}_2 (\vrn, \vm_n, \mathcal{E}_n).
\]
On the other hand, since $\mathcal{A}_n \toMo \mathcal{A}$, we deduce
\[
\liminf_{n \to \infty} \mathcal{F}_2 (\vrn, \vm_n, \mathcal{E}_n) \leq
\min_{\mathcal{A}} \mathcal{F}_2.
\]
This yields the desired conclusion
\[
(\vr, \vm, \mathcal{E}) \in \arg \min_{\mathcal{A}} \mathcal{F}_2.
\]

Finally, we conclude that this process already gives a unique solution. Indeed, we deduce from the strict convexity of the functional in $\vm$ that every minimizer
has the same momentum component. This in turn implies uniqueness of the associated density as the latter can be recovered from the equation of continuity \eqref{w10}.
Finally, we know from {\bf step 1} that the selected solution is minimal with respect to $\prec$.
Accordingly, Proposotion~\ref{mT10} yields
total energy profile
$\mathcal{E}$.

\subsection{Measurability of the selected semigroup}

Finally, we realize that the time averages
\[
\int_0^\infty \theta_{\ep} (\tau - t) \vr(t, \cdot) \dt,
\int_0^\infty \theta_{\ep} (\tau - t) \vm(t, \cdot) \dt,
\]
where $\theta_\ep$ is a family of regularizing kernels converge weakly
\[
\int_0^\infty \theta_{\ep} (\tau - t) \vr(t, \cdot) \dt \to \vr(\tau, \cdot)
\ \mbox{weakly in}\ L^\gamma(\Omega),
\]
\[
\int_0^\infty \theta_{\ep} (\tau - t) \vm(t, \cdot) \dt \to
\vm(\tau, \cdot)\ \mbox{weakly in}\ L^{\frac{2 \gamma}{\gamma + 1}}(\Omega; R^d)
\]
as $\ep \to 0$ for any $\tau \geq$. Similar argument can be used for $\mathcal{E}(\tau-)$ by means of a ``left'' regularization.
We conclude that the solution mapping
\[
(\vr_0, \vm_0, \mathcal{E}_0) \to
(\vr(\tau, \cdot), \vm(\tau, \cdot), \mathcal{E}(\tau-)),
\]
being a pointwise limit of Borel measurable mappings,
is Borel measurable in $L^{\gamma}(\Omega)-{\rm weak} \times
L^{\frac{2 \gamma}{\gamma +1}}(\Td; \Omega)-{\rm weak} \times R$
for any $\tau > 0$. Seeing that the weak and strong topology generate
the same family of Borel sets, we may also conclude the mapping is strongly measurable.

Let us summarize the results obtained.

\begin{mdframed}[style=MyFrame]

\begin{Theorem}[{\bf Semigroup selection}] \label{sT1}	
There is a Borel measurable mapping
\[
\vc{U}: (\vr_0, \vm_0, \mathcal{E}_0) \in \mathcal{D} \to \mathcal{U}[\vr_0, \vm_0, \mathcal{E}_0] \in \mathcal{T}_{\rm weak} \cap \mathcal{T}_{\rm strong}
\]
such that
\[
\vc{U} [\vr_0, \vm_0, \mathcal{E}_0] (t + s) =
\vc{U} \circ [ \vc{U}[\vr_0, \vm_0, \vc{E}_0](s) ] (t) \ \mbox{for any}\ s,t \geq 0.
\]

Moreover, $\vc{U}$ can obtained by a two step selection process:
\[
\vc{U}[\vr_0, \vm_0, \mathcal{E}_0] = \arg \min_{ \left[ \arg \min_{\mathcal{U}[\vr_0, \vm_0, \mathcal{E}_0]} \mathcal{F}_1 \right]} \mathcal{F}_2 ,
\]
where the functionals $\mathcal{F}_1$, $\mathcal{F}_2$ have been specified in \eqref{F1}, \eqref{FF1} or \eqref{FF2}. The selected solutions are admissible, meaning minimal with respect to the order $\prec$.

	\end{Theorem}
	
	\end{mdframed}

\section{Absolute energy minimizers}
\label{A}

Our final goal is to introduce a refined maximal dissipation criterion
to recover a unique weak solution. Consider a solution
$(\vr^1, \vm^1, \mathcal{E}^1) \in \mathcal{U}[\vr_0, \vm_0, \mathcal{E}_0]$ that \emph{is not} maximal. Consequently,
there exists $(\vr^2, \vm^2, \mathcal{E}^2) \in \mathcal{U}[\vr_0, \vm_0, \mathcal{E}_0]$ such that
\[
(\vr^2, \vm^2, \mathcal{E}^2) \prec_{\rm loc}
(\vr^1, \vm^1, \mathcal{E}^1),
\]
meaning
\[
(\vr^2, \vm^2, \mathcal{E}^2)(t) = (\vr^1, \vm^1, \mathcal{E}^1)(t)
\ \mbox{for}\ 0 \leq t \leq T < \infty,
\ \mathcal{E}^2(t) < \mathcal{E}^1(t) \ \mbox{for}\ T < t <  T+\delta.
\]
Next, we compute
\begin{align}
\int_0^\infty \exp(-\lambda t) (\mathcal{E}^2 - \mathcal{E}^1)(t) \dt
&\leq \int_T^{T+ \delta} \exp(-\lambda t) (\mathcal{E}^2 - \mathcal{E}^1)(t) \dt +
\int_{T+ \delta}^\infty \exp(-\lambda t) (\mathcal{E}^2 - \mathcal{E}^1)(t) \dt \br
&\leq \exp(-\lambda(T + \delta)) \left[ \int_T^{T+\delta} (\mathcal{E}^2 - \mathcal{E}^1)(t) \dt + \frac{2}{\lambda} \mathcal{E}_0  \right] < 0,
\label{IN}
\end{align}
provided $\lambda > 0$ is large enough.

This motivates the following definition.

\begin{mdframed}[style=MyFrame]
\begin{Definition}[{\bf Absolute energy minimizer}] \label{wD5}
We say that
\[
(\underline{\vr}, \underline{\vm}, \underline{\mathcal{E}} ) \in
\mathcal{U}[\vr_0, \vm_0, \mathcal{E}_0]
\]
is \emph{absolute energy minimizer} if
\begin{equation} \label{AA1}
\int_0^\infty \exp(-\lambda t) \underline{\mathcal{E}} \dt
\leq \int_0^\infty \exp(-\lambda t) {\mathcal{E}}(t) \dt
\ \mbox{for all}\ \lambda \geq \underline{\lambda}
\end{equation}
for any other solution $(\vr, \vm, \mathcal{E}) \in \mathcal{U}[\vr_0, \vm_0, \mathcal{E}_0]$, where $\underline{\lambda} = \underline{\lambda}(\vr, \vm, \mathcal{E})$ is some number depending on $(\vr, \vm, \mathcal{E})$.
\end{Definition}
\end{mdframed}

It follows from the previous discussion that any absolute energy minimizer
is minimal with respect to the order $\prec_{\rm loc}$; whence an admissible
weak solution of the Euler system.

Next, we report the following variant of Lerch theorem, see \cite[Lemma 5.3]{BreFeiHof19C}.

\begin{Lemma}[{\bf Lerch lemma}] \label{AL2}
Suppose there exists $\underline{\lambda} > 0$ such that
\[
\int_0^\infty \exp(-\lambda t) {\mathcal{E}}^1 \dt
= \int_0^\infty \exp(-\lambda t) {\mathcal{E}^2}(t) \dt
\ \mbox{for all}\ \lambda \geq \underline{\lambda}.
\]

Then
\[
\mathcal{E}^1 = \mathcal{E}^2.
\]	
	\end{Lemma}

It follows from Lemma \ref{AL2} that there exists at most one absolute energy minimizer.
Indeed it is easy to see that the set of all absolute minimizers is convex. Suppose that are two
absolute minimizers $(\underline{\vr}^1, \underline{\vm}^1, \underline{\mathcal{E}})$ and
$(\underline{\vr}^2, \underline{\vm}^2, \underline{\mathcal{E}})$, where
\begin{equation} \label{pom1}
\mathcal{E} = \intO{ E(\vr^1, \vm^1) } = \intO{ E(\vr^2, \vm^2) }.
\end{equation}
The convex combination
\[
(\underline{\vr}, \underline{\vm}, \underline{\mathcal{E}}) =
\frac{1}{2} (\underline{\vr}^1, \underline{\vm}^1, \underline{\mathcal{E}}) +
\frac{1}{2} (\underline{\vr}^2, \underline{\vm}^2, \underline{\mathcal{E}})
\]
is another absolute energy minimizers; whence a weak solution satisfying
\begin{equation} \label{pom2}
\mathcal{E} = \intO{ E \left(\frac{1}{2} \underline{\vr}^1 +
	\frac{1}{2} \underline{\vr}^2, \frac{1}{2} \underline{\vm}^1 +
	\frac{1}{2} \underline{\vm}^2 \right) }.
\end{equation}
As $E$ is strictly convex, the relations \eqref{pom1}, \eqref{pom2} are compatible only if
\[
\underline{\vr}^1 = \underline{\vr}^2,\ \underline{\vm}^1 = \underline{\vm}^2.
\]

Let us summarize the previous discussion.

\begin{mdframed}[style=MyFrame]

	\begin{Theorem}[{\bf Absolute energy minimizers}] \label{AT1}
		Let
		\[
		(\underline{\vr}, \underline{ \vm}, \underline{\mathcal{E}}) \in \mathcal{U}[\vr_0, \vm_0, \mathcal{E}_0]
		\]	
		be an absolute energy minimizer.
		
		Then the following holds true:
		\begin{itemize}
			\item
			$(\underline{\vr}, \underline{ \vm}, \underline{\mathcal{E}})$ is a maximal dissipative solution, in particular
			\begin{equation} \label{A1}
				D_{\underline{\mathcal{E}}} (t) = \underline{\mathcal{E}}(t+) - \intO{ E(\underline{\vr}, \underline{ \vm}) (t, \cdot)} = 0
			\end{equation}
			for any $t \geq 0$. Thus $(\underline{\vr}, \underline{ \vm}, \underline{\mathcal{E}})$ is a weak solution
			of the Euler system, with
			\[
			\underline{\mathcal{E}}(t+) = \intO{E(\underline{\vr}, \underline{\vm)}(t, \cdot)} \ \mbox{for any}\ t \geq 0.
			\]
			
			\item The solution $(\underline{\vr}, \underline{ \vm })(t)$ is right (strongly) continuous in
			\begin{equation} \label{A2}
				L^\gamma(\Omega) \times L^{\frac{2 \gamma}{\gamma + 1}}(\Omega; R^d)
			\end{equation}
			at any $t \geq 0$.
			
			\item
		Each set $\mathcal{U}[\vr_0, \vm_0, \mathcal{E}_0]$ contains at most one absolute energy minimizer.

		\end{itemize}	
	\end{Theorem}
	
\end{mdframed}

Finally, it follows from inequality \eqref{IN} that any \mbox{absolute} minimizer with respect to order $\prec_{\rm loc}$ is an absolute
energy minimizer. Recall that $(\underline{\vr}, \underline{\vm}, \underline{\mathcal{E}})
\in \mathcal{U}[\vr_0, \vm_0, \mathcal{E}_0]$ is an absolute minimizer with respect to
order $\prec_{\rm loc}$ if
\[
(\underline{\vr}, \underline{\vm}, \underline{\mathcal{E}}) \prec_{\rm loc}
(\vr, \vm, \mathcal{E}) \ \mbox{for any}\
(\vr, \vm, \mathcal{E}) \in \mathcal{U}[\vr_0, \vm_0, \mathcal{E}_0], \ (\vr, \vm, \mathcal{E}) \ne
(\underline{\vr}, \underline{\vm}, \underline{\mathcal{E}}).
\]
It is worth noting that the property of being absolute minimizer with respect to $\prec_{\rm loc}$ is a \emph{local} property that can be defined
independently of the actual length of the time interval on which the problem is considered.	

Next, we show that the family of absolute energy minimizers, provided they
exist for any initial data, enjoy the semigroup properties.

\begin{Proposition} \label{AL1}

The following properties hold:

\begin{itemize}

\item
	
Suppose
\[
(\underline{\vr}, \underline{\vm}, \underline{\mathcal{E}}) \in
\mathcal{U}[\vr_0, \vm_0, \mathcal{E}_0]
\]	
is an absolute energy minimizer.

Then its time shift $\mathcal{S}_T
(\underline{\vr}, \underline{\vm}, \underline{\mathcal{E}})$ is
an absolute energy minimizer in $\mathcal{U}[(\underline{\vr}, \underline{\vm}, \underline{\mathcal{E}})(T, \cdot)]$.

\item
Let $(\underline{\vr}^1, \underline{\vm}^1, \underline{\mathcal{E}}^1)$
be an absolute energy minimizer in
$\mathcal{U}[\vr_0, \vm_0, \mathcal{E}_0]$, and
let $(\underline{\vr}^2, \underline{\vm}^2, \underline{\mathcal{E}}^2)$ be an
absolute energy minimizer in $\mathcal{U}[(\underline{\vr}^1, \underline{\vm}^1, \underline{\mathcal{E}}^1)(T, \cdot)]$.

Then their concatenation
\[
(\underline{\vr}^1, \underline{\vm}^1, \underline{\mathcal{E}}^1)
\cup_T (\underline{\vr}^2, \underline{\vm}^2, \underline{\mathcal{E}}^2)
\]
is an absolute energy minimizer in $\mathcal{U}[\vr_0, \vm_0, \mathcal{E}_0]$, meaning
\[
(\underline{\vr}^1, \underline{\vm}^1, \underline{\mathcal{E}}^1) =(\underline{\vr}^1, \underline{\vm}^1, \underline{\mathcal{E}}^1)
\cup_T (\underline{\vr}^2, \underline{\vm}^2, \underline{\mathcal{E}}^2).
\]

\end{itemize}	
	
\end{Proposition}	

\begin{proof}
	
{\bf Shift property.}

Let $(\vr, \vm, \mathcal{E}) \in \mathcal{U}[(\underline{\vr}, \underline{\vm}, \underline{\mathcal{E}})(T, \cdot)]$ be an arbitrary solution.
We have
\begin{align}
\int_0^\infty &\exp(-\lambda t) \mathcal{S}_T \underline{\mathcal{E}}(t) \dt \equiv 	
\int_0^\infty \exp(-\lambda t) \underline{\mathcal{E}}(t + T) \dt \br
&= \exp(\lambda T) \int_0^\infty \exp(-\lambda (t + T)) \underline{\mathcal{E}}(t + T) \dt \br &= \exp(\lambda T) \left[ \int_0^\infty \exp(-\lambda t) \underline{\mathcal{E}}(t) \dt - \int_0^T \exp(- \lambda t) \underline{\mathcal{E}}(t) \dt \right] \br
&\leq \exp(\lambda T) \left[ \int_0^\infty \exp(-\lambda t) [\underline{\mathcal{E}} \cup_T \mathcal{E}](t) \dt - \int_0^T \exp(- \lambda t) \underline{\mathcal{E}}(t) \dt \right] \br
&= \exp(\lambda T) \int_T^\infty \exp(- \lambda t) \mathcal{E}(t - T) \dt  = \int_0^\infty \exp(- \lambda t) \mathcal{E}(t) \dt
\nonumber
\end{align}	
for a suitable sequence $\lambda \geq \underline{\lambda}$ associated to
$(\underline{\vr}, \underline{\vm}, \underline{\mathcal{E}}) \cup_T (\vr, \vm, \mathcal{E})$.	

\medskip

{\bf Concatenation property.}

A straightforward manipulation yields
\begin{align}
\int_0^\infty &\exp (- \lambda t) [\underline{\mathcal{E}}^1
\cup_T \underline{\mathcal{E}}^2](t) \dt  =
\int_0^T \exp(-\lambda t) \underline{\mathcal{E}}^1(t) \dt +
\int_T^\infty \exp(-\lambda t) \underline{\mathcal{E}}^2(t - T) \dt \br
&= \int_0^T \exp(-\lambda t) \underline{\mathcal{E}}^1(t) \dt
+ \exp(\lambda T) \int_0^\infty \exp(- \lambda t) \underline{\mathcal{E}}^2 (t) \dt \br
&\leq \int_0^T \exp(-\lambda t) \underline{\mathcal{E}}^1(t) \dt +
\exp(\lambda T) \int_0^\infty \exp (- \lambda t) \mathcal{S}_T \underline{\mathcal{E}}^1 (t) \dt =
\int_0^\infty \exp(-\lambda t) \underline{\mathcal{E}}^1(t) \dt
\nonumber
\end{align}	
for all $\lambda \geq \underline{\lambda}$ related to $\mathcal{S}_T \underline{\mathcal{E}}^1$. As a matter of fact, it
follows from the previous step and uniqueness of absolute minimizers that
\[
(\vr^2, \vm^2, \mathcal{E}^2) = \mathcal{S}_T (\vr^1, \vm^1, \mathcal{E}^1).
\]
We conclude
\[
\underline{\mathcal{E}}^1
\cup_T \underline{\mathcal{E}}^2 = \underline{\mathcal{E}}^1.
\]
	
	\end{proof}

\section{Conclusion}

We have identified five different classes of dissipative solutions emanating from given data $(\vr_0, \vm_0, \mathcal{E}_0)$ (see Figure \ref{fig}):
\begin{align}
\Big\{ \mbox{absolute energy minimizer} \Big\} &\subset \Big\{ \mbox{maximal dissipative solutions} \Big\} \subset \Big\{\mbox{admissible weak solutions}\Big\} \br
&\subset \Big\{ \mbox{admissible dissipative solutions} \Big\} \subset
\Big\{ \mbox{dissipative solutions} \Big\};
\nonumber
\end{align}
Admissible dissipative solutions do exist for any finite energy initial data and admit  a semigroup selection resulting from a two step process.
Maximal dissipative solutions are admissible weak solution of the Euler system. They comply with Dafermos' principle of maximal dissipation. Their existence for general initial data, however, remains an open problem. The absolute energy minimizer, provided it exists, is a unique weak solution in
$\mathcal{U}[\vr_0, \vm_0, \mathcal{E}_0]$.

\begin{figure}[h!]
\centering\includegraphics[width=130mm]{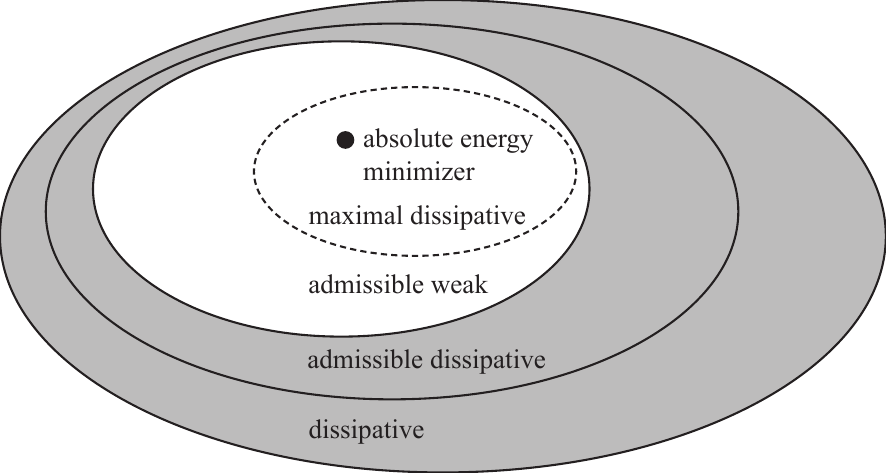}
\caption{Relations between the various solution concepts. The black dot denotes a solution set with a unique element. The gray color indicates solution sets with at least one element, while the existence of elements of the white sets is open. We prove that any maximal dissipative solution is in fact an admissible weak solution.}
\label{fig}
\end{figure}

\section*{Data availability statement} 
Data sharing is not applicable to this article as no
data sets were generated  during the current study.

\section*{Conflict of interest}
The authors have no relevant financial or non-financial interests to disclose.

\def\cprime{$'$} \def\ocirc#1{\ifmmode\setbox0=\hbox{$#1$}\dimen0=\ht0
	\advance\dimen0 by1pt\rlap{\hbox to\wd0{\hss\raise\dimen0
			\hbox{\hskip.2em$\scriptscriptstyle\circ$}\hss}}#1\else {\accent"17 #1}\fi}


\end{document}